\documentclass[11pt]{amsart}

\usepackage[margin=1in]{geometry}
\usepackage{setspace}
\usepackage{amssymb}
\usepackage[T1]{fontenc}
\usepackage{lmodern}
\usepackage{mathtools}
\usepackage{mathrsfs}
\usepackage[all]{xy}
\usepackage[headings]{fullpage}
\usepackage{paralist}
\usepackage{xspace}
\usepackage{tikz}
\usepackage{tikz-cd}
\usepackage{color}
\usepackage[normalem]{ulem}
\usepackage[colorlinks=true,linkcolor=blue, citecolor=blue, urlcolor=blue, pagebackref=true]{hyperref}
\usepackage{enumitem}
\usepackage{graphicx}
\def\chk#1{#1^{\smash{\scalebox{.7}[1.4]{\rotatebox{90}{\textnormal{\guilsinglleft}}}}}}

\numberwithin{equation}{section}
\numberwithin{figure}{section}
\theoremstyle{plain}
\newtheorem{theorem}[equation]{Theorem}

\newtheorem{corollary}[equation]{Corollary}
\newtheorem{proposition}[equation]{Proposition}

\theoremstyle{definition}

\newtheorem{remark}[equation]{Remark}

\newtheorem{example}[equation]{Example}

\newtheorem{definition}[equation]{Definition}

\newtheorem{notation}[equation]{Notation}

\newtheorem{discussion}[equation]{Discussion}
\newenvironment{discussionbox}[1][]{%
    \begin{discussion}[#1]\pushQED{\qed}}{\popQED \end{discussion}}

\newtheorem{observation}[equation]{Observation}

\newtheorem{construction}[equation]{Construction}

\DeclareMathOperator{\Schur}{S}
\DeclareMathOperator{\homology}{H}
\makeatletter
\def\RDerChar{\mathbf{R}}
\def\RDer{\@ifnextchar[{\R@Der}{\ensuremath{\RDerChar}}}
\def\R@Der[#1]{\ensuremath{\RDerChar^{#1}}}
\def\theenumi{\@alph\c@enumi}

\def\vdots{\vbox{\baselineskip=3pt \lineskiplimit=0pt 
\kern6pt \hbox{.}\hbox{.}\hbox{.}}} 
\makeatother
\DeclareMathOperator{\Grass}{G}
\newcommand{\GL}{\mathrm{GL}}
\newcommand{\SL}{\mathrm{SL}}
\newcommand{\SP}{\mathrm{Sp}}
\newcommand{\Id}{\mathrm{Id}}
\newcommand{\Mat}{\mathrm{Mat}}
\newcommand{\Sym}{\mathrm{Sym}}
\def\JDiag[#1]{\mathrm{J}}

\title{Free resolutions of some Schubert singularities in the Lagrangian Grassmannian}
\dedicatory{Dedicated to the memory of Professor Robert Steinberg}
\author{Reuven Hodges}
\address{Northeastern University, Boston, Massachusetts. USA.}
\email{hodges.r@husky.neu.edu}

\author{V. Lakshmibai}
\address{Northeastern University, Boston, Massachusetts. USA.}
\email{lakshmibai@neu.edu}

\thanks{The second author was supported by NSA grant H98230-11-1-0197, NSF grant 
0652386.}

\begin{document}

\begin{abstract}

In this paper we construct free resolutions of certain class of closed subvarieties of affine space of symmetric matrices (of a given size). Our class covers the symmetric determinantal varieties (i.e.,  determinantal varieties in the space of symmetric matrices), whose resolutions were first constructed by  J{\'o}zefiak-Pragacz-Weyman\cite{JPWResolutionsDeterminental81}.  Our approach follows the techniques developed in \cite{KLPSResolutionSchSing15}, and uses the geometry of Schubert varieties.  
\end{abstract}
\maketitle

\section{Introduction}
\label{sec:intro}
\numberwithin{equation}{section}

This paper is a sequel to \cite{KLPSResolutionSchSing15}. In \cite {LascouxSyzDet78}, Lascoux constructed a minimal free resolution of the coordinate ring of the determinantal varieties (consisting of $m \times n$ matrices (over $\mathbb{C}$) of rank at most $k$, considered as a closed subvariety of the $mn$-dimensional affine space of all $m \times n$ matrices), as a
module over the $mn$-dimensional polynomial ring (the coordinate ring of the the $mn$-dimensional affine space).

In \cite{KLPSResolutionSchSing15}, the authors construct free resolutions for a larger class of singularities which can be realized as the intersection of a singular Schubert variety and the opposite big cell inside a Grassmannian.

In \cite{JPWResolutionsDeterminental81}, J{\'o}zefiak-Pragacz-Weyman constructed a minimal free resolution of the coordinate ring of the determinantal varieties (in the space of symmetric matrices) as a module over the coordinate ring of the space of symmetric matrices. In this paper we construct free resolutions for a certain class of closed subvarieties of the affine space of symmetric matrices, which includes the symmetric determinantal varieties. The  technique adopted in \cite{KLPSResolutionSchSing15} is algebraic group-theoretic, and we follow this approach. 

We now describe the results of this paper. Let $n$ be a positive integer. Let $V=\mathbb{C}^{2n}$ and let $(\cdot,\cdot)$ be a non-degenerate skew-symmetric bilinear form on $V$. Let $H=\SL(V)$ and $G=\SP(V)=$ $\left\{ Z\in \SL(V)\,\,|\,\, Z\mbox{ leaves the form }(\cdot,\cdot)\mbox{ invariant}\right\}$. We take the matrix of the form, with respect to the standard basis of $V$, to be

\begin{center}
$F=\left[\begin{array}{cc}
0 & \JDiag[n]\\
-\JDiag[n] & 0
\end{array}\right]$
\par\end{center}

\noindent where $\JDiag[n]$ is the anti-diagonal $(1,...,1)$, in this case of size $n$. To simplify our notation we will normally omit specifying the size of $J$ as it will be obvious from the context. We may realize $\SP(V)$ as the fixed point set of the involution $\sigma\,:\, H\rightarrow H$ given by $\sigma(Z)=F(Z^{T})^{-1}F^{-1}$ (cf. \cite{SteinbergLectChev68}).

Denoting by $T_{H}$ and $B_{H}$ the maximal torus in $H$ consisting of diagonal matrices and the Borel subgroup in $H$ consisting of upper triangular matrices, respectively, we have that $T_{H}$ and $B_{H}$ are stable under $\sigma$ and we set $T_{G}=T_{H}^{\sigma}$, $B_{G}=B_{H}^{\sigma}$. It is easily checked that $T_{G}$ is a maximal torus in $G$ and $B_{G}$ is a
Borel subgroup in $G$.

Thus we obtain $W_G\hookrightarrow W_H$ where $W_G, W_H$ denote the Weyl groups of $G,H$ respectively (with respect to $T_G,T_H$ respectively). Further, $\sigma$ induces an involution on $W_H$ such that $W_G=W_H^\sigma$. Thus we obtain $$W_G=\{(a_1 \cdots a_{2n}) \in S_{2n} \mid a_i=2n+1-a_{2n+1-i},\ 1 \leq i \leq 2n \}, $$ where $S_{2n}$ is the symmetric group on $2n$ letters. Thus $w=(a_1\cdots a_{2n}) \in W_G$ is known once $(a_1 \cdots a_n)$ is known.  We shall denote an element $(a_1\cdots a_{2n})$ in $ W_G$ by just $(a_1\cdots a_n)$. Further, for $w\in W_G$, denoting by $X_G(w)$ (respectively $X_H(w)$), the the associated Schubert variety in $G/B_G$ ($H/B_H$), we have that under the canonical inclusion $G/B_G\hookrightarrow H/B_H$,  $X_G(w)=X_H(w)\cap G/B_G$,  scheme-theoretically.

Let $P=P_{\widehat{n}}$, the maximal parabolic subgroup of $G$ corresponding to omitting the simple root $\alpha_n$, the set of simple roots of $G$ being indexed as in \cite{BourbakuEMGAL68}. Let $1\leq k<r\leq n$ be positive integers, and let $w = \mathcal{W}_{k,r}$ (cf. \eqref{notation:parabolicSubgroupsStr}). Our main result (cf. Theorem \ref{theorem:subbundleSp2n}) is a description of the minimal free resolution of the coordinate ring of $Y_P(w) := X_{P}(w) \cap O^{-}_{G/{{P}}}$, the \emph{opposite cell} of $X_{{P}}(w)$, as a module over the coordinate ring of $O^-_{G/P}$. For this, as in \cite{KLPSResolutionSchSing15}, we use the Kempf-Lascoux-Weyman ``geometric technique'' for constructing minimal free resolutions. 

Given a closed subvariety $Y$ of an affine space $\mathbb{A}$, the geometric technique constructs a minimal free resolution of the coordinate ring $\mathbb{C}[Y]$ of $Y$ as a module over the coordinate ring $R$ of $\mathbb{A}$ ($R$ is a polynomial ring). This is achieved by first finding a projective variety $V$ such that the following commutative diagram can be constructed
\begin{equation}
\label{eq:geomtricTech}
\begin{tikzcd}
Z \arrow[hookrightarrow]{r} \arrow[d, "q'"]
& \mathbb{A}  \times V \arrow[d, "q"] \arrow[r] & V\\
Y \arrow[hookrightarrow]{r} 
& \mathbb{A}
\end{tikzcd}
\end{equation}
where $q$ is projection to the first coordinate, and its restriction $q'$ is a proper, birational map from a subbundle $Z$ of the trivial bundle $\mathbb{A} \times V$. Finding a projective variety $V$ such that this diagram may be constructed is in general not trivial, but in our case it is made possible by exploiting the geometry of the Schubert varieties and their opposite cells. Denote the dual of the quotient bundle on $V$ associated to the subbundle $Z$ by $\xi$. Let $\RDer q'_* \mathscr{O}_{Z}$ be the right derived image of the structure sheaf $\mathscr{O}_{Z}$ of $Z$. If the map $q'$ has the property that the induced map $\mathscr{O}_Y \to \RDer q'_* \mathscr{O}_{Z}$ is a quasi-isomorphism then the complex $F_{\bullet}$ defined below is a minimal free resolution of $\mathbb{C}[Y]$.
\begin{center}
$\displaystyle F_i = \bigoplus_{j \geq 0} \homology^j(V, \bigwedge^{i+j} \xi)
\otimes_\mathbb{C} R(-i-j)$
\end{center}
In this expression $R(-i-j)$ refers to shifting the natural grading of the coordinate ring $R$. Note that if $Y$ has rational singularities and the map $q'$ is a desingularization, then the induced map $\mathscr{O}_Y \to \RDer q'_* \mathscr{O}_{Z}$ is guaranteed to be a quasi-isomorphism. Then the minimal free resolution of $\mathbb{C}[Y]$ can be computed so long as the vector bundle cohomology groups in the above expression can be computed.

We employ the geometric technique for the opposite cells $Y_P(w)$, $w = \mathcal{W}_{k,r}$ for $1 \leq k < r \leq n$. To begin, we must first construct the commutative diagram \eqref{eq:geomtricTech} in the case where $Y = Y_P(w)$ and $\mathbb{A} = O^-_{G/P}$.  Let $\tilde{P}$ be the two-step parabolic subgroup $P_{\widehat{r-k},\widehat{n}}$ of $G$, and let $\tilde{w}$ be the minimal representative of $w\tilde{P}$ in $W^{\tilde{P}}$ (that is, the set of minimal coset representatives in $W$, under the Bruhat order, of $W/W_{\tilde{P}}$, where $W_{\tilde{P}}$ is the Weyl group of $\tilde{P}$). Let $w^{\prime}:=(k+1,..,r,n,..,r+1,k,..,1)\in S_n$, the Weyl group of $\GL_n$. Let $Z_{\tilde{P}}(\tilde{w}):= Y_{P}(w)\times_{X_{P}(w)}X_{\tilde{P}}(w) (=\,(O_{G/P}^{-}\times P/\tilde{P})\cap X_{\tilde{P}}(w)\,)$. Then it turns out that $Z_{\tilde{P}}(\tilde{w})$ is smooth (cf. \ref{definition:ZPW}), and is a desingularization of $Y_{P}(w)$. Let $p$ be the map given by the composition $Z_{\tilde{P}}(\tilde{w})\hookrightarrow O_{G/P}^{-}\times P/\tilde{P}\rightarrow P/\tilde{P}$ where the first map is inclusion and the second map is the projection to the second coordinate. We show, in Theorem \ref{theorem:subbundleSp2n}, that $p$ identifies $Z_{\tilde{P}}(\tilde w)$ as a sub-bundle of  the trivial bundle $O_{G/P}^{-} \times X_{P'_{\widehat{r-k}}}(w^{\prime})$ over $X_{P'_{\widehat{r-k}}}(w^{\prime})$, which arises as the restriction (to $X_{P'_{\widehat{r-k}}}(w^{\prime})$) of a certain homogeneous vector-bundle on $\GL_n/P'_{\widehat{r-k}}$. We may then realize the commutative diagram required by the geometric technique as
\begin{equation}
\label{eq:geomtricTechYpW}
\begin{tikzcd}
Z_{\tilde P}(\tilde w) \arrow[hookrightarrow]{r} \arrow[d, "q'"]
& O^-_{G/P} \times X_{P'_{\widehat{r-k}}}(w^{\prime}) \arrow[d, "q"] \arrow[r] 
& X_{P'_{\widehat{r-k}}}(w^{\prime})\\
Y_{P}(w) \arrow[hookrightarrow]{r} 
& O^-_{G/P}
\end{tikzcd}
\end{equation}
Further, as mentioned above, the restriction $q'$ of the projection map $q:O^-_{G/P} \times X_{P'_{\widehat{r-k}}}(w^{\prime}) \rightarrow O^-_{G/P}$ to the subbundle $Z_{\tilde P}(\tilde w)$ is a desingularization of $Y_{P}(w)$. This, in addition to the fact that Schubert varieties have rational singularities, implies that in this case $F_{\bullet}$ is a minimal resolution of the coordinate ring of $Y_{P}(w)$. Thus to complete our result we must now compute the cohomology groups of the vector bundles $\bigwedge^{t} \xi$ on $X_{P'_{\widehat{r-k}}}(w^{\prime})$ for arbitrary $t$. As is easily seen, $X_{P'_{\widehat{r-k}}}(w^{\prime})$ is a Grassmannian, namely, $\GL_r/P''_{\widehat{r-k}}$; further, the bundles $\wedge^t \xi$ (on $\GL_r/P''_{\widehat{r-k}}$) are homogeneous, but are not of Bott-type, namely, they are not completely reducible (so one can not apply Bott-algorithm for computing the cohomology). This can be resolved in two ways. In \cite {OttavianiRubeiQuiversCohVB06} the authors determine the cohomology of general
homogeneous bundles on Hermitian symmetric spaces, and thus their results can be used to determine $H^{\bullet}(V,\wedge^t \xi)$. Alternatively, using a technique from \cite{WeymanCohVBSyz03}, we may compute the resolution of a related space (whose associated homogeneous vector bundle is of Bott-type) from which we retrieve the resolution of the coordinate ring of $Y_P(w)$ as a subcomplex.

We hope to extend the results of this paper to Schubert varieties in the orthogonal Grassmannian. Details will appear in a subsequent paper.

The sections of this paper are laid out as follows. In Section \ref{sec:prelim} the required background, notation, and results on Schubert varieties in the full and partial flag varieties, and the symplectic flag variety are recalled. Then homogeneous vector bundles on flag varieties are discussed, as well as the Bott-algorithm for computing their cohomology. In  Section \ref{sec:desing} we prove the results necessary for constructing the commutative diagram in \eqref{eq:geomtricTechYpW}, and show that $q':Z_{\tilde P}(\tilde w)\rightarrow Y_{P}(w)$ is a desingularization. In Section \ref{sec:freeresolutions} we briefly recall some additional details on the geometric technique. Finally, Section \ref{sec:cohomologyHVBStepTwo} explains two methods for computing the cohomology groups of the homogeneous bundles found in the complex $F_{\bullet}$ in our particular case.

\section{Preliminaries}
\label{sec:prelim}
\numberwithin{equation}{subsection}

In this section we recall notation and results for Schubert varieties in the flag variety and symplectic flag variety. We also review homogeneous vector bundles and the Bott-algorithm for computing their cohomology.

\subsection{Background for Type A}
\label{sec:Preliminaries:BackgroundA}
For proofs and a more in depth introduction see \cite{LakRaghSMT08}.

Let $\GL_N$ be the group of invertible $N\times N$ matrices with entries in $\mathbb{C}$. Let $B_N$ be the subgroup of upper triangular matrices in $\GL_N$, with $B_N^{-}$ the subgroup of lower triangular matrices. We denote by $\mathrm{diag}(t_1,\ldots,t_N)$ the diagonal matrix in $\GL_N$ with entries $t_i \in \mathbb{C}^{\times}$ along the main diagonal. The subgroup of $\GL_N$ consisting of all diagonal matrices will be denoted $T_N$. The \emph{parabolic subgroups} of $\GL_N$ containing $B_N$ may be constructed as follows. For $1 \leq d < N$ we define the maximal parabolic subgroup $P_{\widehat{d}}$ of $\GL_N$ to consist of all invertible $N \times N$ matrices with a block of zeroes of size $N-d \times d$ in the bottom left. Then there is a one-to-one correspondence between subsets of $\{1,\ldots,N\}$ and parabolic subgroups of $\GL_N$ containing $B_N$; for $A \subseteq \{1,\ldots,N\}$ the corresponding parabolic is $P_{\widehat{A}}:=\bigcap_{i\in A} P_{\widehat{d}}$.

The \emph{character group} of $T$ is $X^{*}(T):= \mathrm{Hom}(T,\mathbb{C}^{\times})$. If $\epsilon_i$ is the character that maps the diagonal matrix $\mathrm{diag}(t_1,\ldots,t_N)$ to $t_i$ in $\mathbb{C}$, then the set $\{ \epsilon_i | 1 \leq i \leq N \}$ generates $X^{*}(T)$. The \emph{root system} of $\GL_N$ is the set $R=\{\epsilon_i - \epsilon_j | 1 \leq i,j \leq N \}$, where $\epsilon_i - \epsilon_j$ is the element in $X^{*}(T)$ that maps $\mathrm{diag}(t_1,\ldots,t_N)$ to $t_i t_j^{-1}$ in $\mathbb{C}$. The positive roots of $\GL_N$ are $R^{+}=\{\epsilon_i - \epsilon_j | 1 \leq i < j \leq N \}$, with the negative roots $R^{-}=\{\epsilon_i - \epsilon_j | 1 \leq j < i \leq N \}$. The simple roots of $\GL_N$ are $S = \{\alpha_i := \epsilon_i - \epsilon_{i+1} | 1 \leq i < N \}$. 

The \emph{Weyl group} of $\GL_N$ is generated by the simple reflections $s_i$ for $1 \leq i < N$. We have that $W$ is isomorphic as a group to the symmetric group $S_N$ under the map which sends $s_i$ to the transposition $(i,i+1)$. The one-line notation for elements of $W$ is a sequence $(i_1,\ldots,i_N)$; this corresponds to the permutation that sends $n$ to $i_n$. The length of an element $w \in W$ is the minimal non-negative integer $m$ such that $w=s_{j_1}\cdots s_{j_m}$ in $W$. We define $W_{P_A}$ to be the subgroup of $W$ generated by the simple reflections $s_i$, $i \in A$. Then $W^{P_A}$ is the set consisting of minimal coset representatives under the Bruhat order of $W/W_{P_A}$. The set $W^{P_A}$ can be shown to be in bijection with the set $\{ (a_1,\ldots,a_d) | 1 \leq a_1 < a_2 < \cdots < a_d \leq N \}$.

We refer to $\GL_N / B_N$ as the \emph{full flag variety} and for all parabolic subgroups $P_{\widehat{A}}$, $\GL_N / P_{\widehat{A}}$ is called the \emph{partial flag variety}. The \emph{Grassmannian} $\Grass_{d,N}$ of $d$-dimensional subspaces of $\mathbb{C}^N$ can be identified with $\GL_N / P_{\widehat{d}}$.

Fix a parabolic subgroup $P$. Let $R_P$, $R_P^+$, and $R_P^-$ be the subset of roots, positive roots, and negative roots corresponding to $P$, respectively. For $w \in W$, we define $e_w$ to be the coset $w P$. The \emph{Schubert variety} $X_{P}(w)$ is the $B$-orbit closure of $e_w$ in $\GL_N / P$, with the canonical reduced scheme structure. For concreteness, whenever we write $X_{P}(w)$ we will always mean that $w$ corresponds to the representative of its coset in $W^P$. The $B_N^{-}$-orbit of $e_{id}$ is a dense open subset of $\GL_N / P$ called the \emph{opposite big cell in $\GL_N / P$} and denoted $O^{-}_{\GL_N/ P}$. For a Schubert variety $X_{P}(w)$ in $\GL_N / P$ we define the \emph{opposite cell in $X_{P}(w)$} to be $Y_{\tilde{P}}(w) := X_{P}(w) \cap O^{-}_{\GL_N/P}$. The opposite cell $Y_{\tilde{P}}(w)$ is an open affine subvariety of $X_{P}(w)$ containing $e_{id}$. 

\subsection{Background for Type C}
\label{sec:Preliminaries:BackgroundC}

Below we review the properties of symplectic Schubert varieties relevant to this paper. For a more in-depth introduction the reader may refer to \cite[Chapter 6]{LakRaghSMT08}.

Let $n$ be a positive integer. Let $V=\mathbb{C}^{2n}$
and let $(\cdot,\cdot)$ be a non-degenerate skew-symmetric bilinear form on $V$. Let $H=\SL(V)$ and $G=\SP(V)=$ $\left\{ Z\in \SL(V)\,\,|\,\, Z\mbox{ leaves the form }(\cdot,\cdot)\mbox{ invariant}\right\} $.
We take the matrix of the form, w.r.t the standard basis of $V$,
to be

\begin{center}
$F=\left[\begin{array}{cc}
0 & \JDiag[n]\\
-\JDiag[n] & 0
\end{array}\right]$
\par\end{center}

\noindent where $\JDiag[n]$ is the anti-diagonal $(1,...,1)$, in this case of size $n$.
To simplify our notation we will normally omit specifying the size of $J$ as it will be obvious from the context. We may realize $\SP(V)$ as the fixed point set
of the involution $\sigma\,:\, H\rightarrow H$ given by $\sigma(Z)=F(Z^{T})^{-1}F^{-1}$ (cf. \cite{SteinbergLectChev68}).
That is

\begin{center} 
$\begin{alignedat}{2}G 
& = \{Z\in \SL(V)\,\,|\,\, Z^{T}FZ=F\}\\  
& = \{Z\in \SL(V)\,\,|\,\, F^{-1}(Z^{T})^{-1}F=Z\}\\ 
& = \{Z\in \SL(V)\,\,|\,\, F(Z^{T})^{-1}F^{-1}=Z\}\\ 
& = H^{\sigma} \end{alignedat} $
\par\end{center}

Denote by $T_{H}$ and $B_{H}$ the maximal torus
in $H$ consisting of diagonal matrices and the Borel subgroup in
$H$ consisting upper triangular matrices, respectively. It is easily
seen that $T_{H}$ and $B_{H}$ are stable under $\sigma$ and we
set $T_{G}=T_{H}^{\sigma}$, $B_{G}=B_{H}^{\sigma}$. It is easily
checked that $T_{G}$ is a maximal torus in $G$ and $B_{G}$ is a
Borel subgroup in $G$.

Thus we obtain $$W_G\hookrightarrow W_H$$ where $W_G, W_H$ denote the Weyl groups of $G,H$ respectively (with respect to $T_G,T_H$ respectively).
 Further, $\sigma$ induces an involution on $W_H$:
 $$w=(a_1,\cdots,a_{2n})\in W_H,
 \sigma(w)=(c_1,\cdots,c_{2n}),\,c_i=2n+1-a_{2n+1-i}$$ and
 $$W_G=W_H^\sigma$$Thus we obtain
$$W_G=\{(a_1 \cdots a_{2n}) \in S_{2n} \mid a_i=2n+1-a_{2n+1-i},\ 1 \leq i
\leq 2n \},
$$ where $S_{2n}$ is the symmetric group on $2n$ letters. Thus $w=(a_1\cdots a_{2n}) \in W_G$ is known once $(a_1 \cdots
a_n)$ is known.  We shall denote an element $(a_1\cdots a_{2n})$
in $ W_G$ by just $(a_1\cdots a_n)$. For example, $(4231) \in
S_{4}$ represents $(42) \in W_{G}$ for $G=\SP(4)$.

The involution $\sigma$
induces an involution on $X^{*}(T_H)$, the character group of $T_H$:
$$\chi\in X^{*}(T_H),\,\sigma(\chi)(D)=\chi(\sigma(D)),\,D\in T_H .$$
Let $\epsilon_i, 1\le i\le 2n$ be the character in $X^{*}(T_H)$ that maps
$\mathrm{diag}(t_1,\ldots,t_{2n}) \in T_H$ to $t_i$ in $\mathbb{C}$. We have
$$\sigma(\epsilon_i)=-\epsilon_{2n+1-i}$$ Now it is easily seen
that the under the canonical surjective map
$$\varphi:X^{*}(T_H)\rightarrow X^{*}(T_G)$$ we have
$$\varphi(\epsilon_i)=-
\varphi(\epsilon_{2n+1-i}),\,1\le i\le 2n .$$ Let
$R_H:=\{\epsilon_i-\epsilon_j,1\le i,j\le 2n\}$, the root system
of $H$ relative to $T_H$, and
$R^+_H:=\{\epsilon_i-\epsilon_j,1\le i<j\le 2n\}$, the set of
positive roots relative to $B_H$. We have the following:
\begin{enumerate}
\item $\sigma$ leaves $R_H$ (respectively $R^+_H$) stable. 
\item For $\alpha,\beta\in R_H,\varphi(\alpha)=\varphi(\beta)\Leftrightarrow \alpha=\sigma(\beta).$ \item $\varphi$ is equivariant for the canonical action of $W_G$ on $X^{*}(T_H),X^{*}(T_G)$. 
\item $R_H^\sigma=\{\pm(\epsilon_i-\epsilon_{2n+1-i}),1\le i\le n\}.$
\end{enumerate}

Let $R_G$ be the set of roots of $G$ with respect to
$T_G$ and $R_G^+$ the set of positive roots with respect to $B_G$.
Using the above facts and the explicit nature of the adjoint
representation of $G$ on Lie$\,G$, we deduce that
$$R_G=\varphi(R_H),\,R^+_G=\varphi(R^+_H) .$$ In particular, $R_G$ (respectively $R^+_G$) gets
identified with the orbit space of $R_H$ ($R^+_H$) modulo
the action of $\sigma$. Thus we obtain the following
identification:
\begin{equation*}\label{e:roots.c}
R_G=\{ \pm (\epsilon_i \pm \epsilon_j),\ 1\leq i<j\leq n\} \cup
\{\pm 2\epsilon_i,\ i=1,...,n  \}
\end{equation*}
\begin{equation*}\label{e:pos.roots.c}
R_G^+ = \{  (\epsilon_i \pm \epsilon_j),\ 1\leq i<j\leq n\} \cup
\{ 2\epsilon_i,\ i=1,...,n  \}.
\end{equation*}
The set $S_G$ of simple roots in $R_G^+$ is given by
\begin{equation*}\label{e:sim.roots.c}
 S_G:=\{ \alpha_{i}=\epsilon_i - \epsilon_{i+1},\ 1 \leq i \leq n-1\}
\cup \{\alpha_{n}=2\epsilon_n  \}.
\end{equation*}
 Let us denote the simple reflections in $W_G$ by $\{ s_i,\  1
\leq
 i \leq n \}$, namely, $s_i$ is the reflection with respect to
$\epsilon_i - \epsilon_ {i+1},\  1 \leq i \leq n-1$, and $s_n$ is
reflection with respect to $2\epsilon_n.$ Then we have
\begin{equation}
\label{equation:simpleReflectionsSP2n}
s_i=\begin{cases}r_ir_{2n-i}, &\text{if}\ 1 \leq i \leq n-1\\
                 r_n,             &\text{if}\  i=n
\end{cases}
\end{equation}

where $r_i$ denotes the transposition $(i,i+1)$ in $S_{2n},\   1
\leq i \leq 2n-1.$  

For $w\in W_G$, let us denote $l(w,W_H)$ (resp. $l(w,W_G)$), the length of $w$ as an element of $W_H$ (resp. $W_G$). For $w=(a_1,\cdots,a_{2n})\in W_H$, denote
\begin{equation}
\label{equation:mw}
m(w):=\#\{i\le n\,|\,a_i>n\}.
\end{equation}

\noindent Then for $w=(a_1,\cdots,a_{2n})\in W_G$, we have $l(w,W_G)={\frac{1}{2}}(l(w,W_H)+m(w)).$

\begin{proposition}
\label{proposition:inducedSchemeTheSchubert}
\cite[Proposition 6.2.5.1]{LakRaghSMT08}Let $w\in W_G$. Let $X_G(w)$ (respectively $X_H(w)$) be the associated
Schubert variety in $G/B_G$ ($H/B_H$). Under the canonical inclusion $G/B_G\hookrightarrow H/B_H$, we have $X_G(w)=X_H(w)\cap
G/B_G$. Further, the intersection is scheme-theoretic.
\end{proposition}

\begin{notation}
\label{notation:parabolicSubgroupsPre}
For the remainder of the paper we fix the following notation. Let $1\leq k<r\leq n$ be positive integers. Let $Q=Q_{\hat{n}}$ to be the parabolic subgroup of $H$ corresponding to omitting the root $\alpha_{n}$ and $P=P_{\hat{n}}$ to be the parabolic subgroup of $G$ corresponding to omitting the root $\alpha_{n}$. Let $\tilde{P}$ be the two-step parabolic subgroup $P_{\widehat{r-k},\widehat{n}}:=P_{\widehat{r-k}} \cap P_{\widehat{n}}$ of $G$. Let $\tilde{Q}$ be the three step parabolic subgroup $Q_{\widehat{r-k},\widehat{n},\widehat{2n-(r-k)}}:=Q_{\widehat{r-k}}\cap Q_{\widehat{n}} \cap Q_{\widehat{2n-(r-k)}}$ in $H$. Note that $P=Q^{\sigma}$ and $\tilde{P}=\tilde{Q}^{\sigma}$. Finally, we will identify $P/\tilde{P}$ with $\GL_n/P'_{\widehat{r-k}}$ where $P'_{\widehat{r-k}}$ is the parabolic subgroup of $\GL_n$ omitting the root $\alpha_{r-k}$.
\end{notation}

\begin{definition}
\label{definition:persymmetric}
A square $m\times m$ matrix $X$ is \emph{persymmetric} if $\JDiag[m]X=X^{T}\JDiag[m]$. Or equivalently if $\JDiag[m]X$ is symmetric.
\end{definition}

\begin{remark}
\label{remark:identificationOMinusGP}
We denote by $\Mat_{n}$ the space of $n\times n$ matrices. Let $K$ be the subgroup of $H$ consisting of matrices of the form

\begin{center}
$\left[\begin{array}{cc}
\Id_{n} & 0\\
Y & \Id_{n}
\end{array}\right],\, Y\in \Mat_{n}$
\par\end{center}
The canonical morphism $H\rightarrow H/Q$ induces a morphism $\psi_H:K\rightarrow H/Q$. We have that $\psi_H$ is an open immersion, and $\psi_H(K)$ gets identified with the \emph{opposite big cell} $O^-_{H/Q}$ in $H/Q$.

\noindent $O_{H/Q}^{-}$ is $\sigma$-stable and we can identify the \emph{opposite big cell} $O_{G/P}^{-}$ as
\begin{center}
$O_{G/P}^{-}= (O_{H/Q}^{-})^{\sigma}=\left\{ z\in K\,|\, \JDiag[n]Y^{T}\JDiag[n]=Y\right\}$\cite[Corollary 6.2.4.3]{LakRaghSMT08}.
\end{center}
So $O_{G/P}^{-}$ is the subspace of $K$ with $Y$ persymmetric. Thus we can identify $O_{G/P}^{-}$ with the space of symmetric $n\times n$ matrices, $\Sym_n$, under the map $O_{G/P}^{-}\to \Sym_n$ given by 
\begin{center}
$\left[\begin{array}{cc}
\Id_{n} & 0\\
Y & \Id_{n}
\end{array}\right]\mapsto JY$.
\end{center}
\end{remark}

\subsection{Opposite Cells in Schubert Varieties in the Symplectic Flag Variety}
\label{sec:Preliminaries:SchubertVarieitiesSP2n}
A matrix $z\in \SL(V)$ with $n\times n$ block form 
\begin{center} $\left[\begin{array}{cc}
A_{n\times n} & C_{n\times n}\\
D_{n\times n} & E_{n\times n}
\end{array}\right]$
\end{center} \noindent is an element of $G$ if and only if $z^{T}Fz=F$,
i.e. if and only if the following conditions hold on the $n\times n$
blocks.\begin{equation}
\label{equation:StructureSP2nAD}
A^{T}\JDiag[n]D=D^{T}\JDiag[n]A
\end{equation}\begin{equation}
\label{equation:StructureSP2nCE}
C^{T}\JDiag[n]E=E^{T}\JDiag[n]C
\end{equation}\begin{equation}
\label{equation:StructureSP2nADCE}
\JDiag[n]=(A^{T}\JDiag[n]E-D^{T}\JDiag[n]C)=(E^{T}\JDiag[n]A-C^{T}\JDiag[n]D)
\end{equation} The following proposition will prove useful throughout the rest of the paper.

\begin{proposition}
\label{proposition:structureOMinusSp2n}
Let $U_{P}^{-}$ be the negative unipotent radical of the parabolic subgroup $P=P_{\widehat{n}}$.
\begin{asparaenum}
\item \label{proposition:structureOMinusSp2nUminus}
\noindent The is a natural identification of $O_{G/P}^{-}$ with $U_{P}^{-}P/P$

\item \label{proposition:structureOMinusSp2ninvertibleA}
\noindent Let
\begin{center}
$z=\left[\begin{array}{cc}
A_{n\times n} & C_{n\times n}\\
D_{n\times n} & E_{n\times n}
\end{array}\right]$
\par\end{center}

be an element in $G$. Then $zP\in O_{G/P}^{-}$ if and only if
the $n\times n$ matrix $A$ is invertible.

\noindent \item \label{proposition:structureOMinusSp2nInverseImage}
\noindent Let $\pi:G/\tilde{P}\rightarrow G/P$ be the natural induced quotient map. Then $\pi^{-1}(O_{G/P}^{-}) \cong O_{G/P}^{-}\times P/\tilde{P}$. The elements of $O_{G/P}^{-}\times P/\tilde{P}$ are all of the form
\begin{center}
$\left[\begin{array}{cc}
A_{n\times n} & 0\\
D_{n\times n} & \JDiag[n](A^{T})^{-1}\JDiag[n]
\end{array}\right] (\textrm{mod}\,\tilde{P})\in G/\tilde{P}$.
\par\end{center}
with two matrices
\[
\begin{bmatrix}
A_{n \times n} & 0_{n \times n} \\
D_{n \times n} & \JDiag[n](A^{T})^{-1}\JDiag[n]
\end{bmatrix} \;\textrm{and}\;
\begin{bmatrix}
A'_{n \times n} & 0_{n \times n} \\
D'_{n \times n} & \JDiag[n](A'^{T})^{-1}\JDiag[n]
\end{bmatrix}
\]
being equivalent $(\textrm{mod}\,\tilde{P})$ if and only if there exists a matrix $q \in P'_{\widehat{r-k}}$ (as defined in Notation \ref{notation:parabolicSubgroupsPre}) such that $A' = Aq$ and $D' = Dq$.

\noindent \item \label{proposition:structureOMinusSp2nSLnModQ}

\noindent The projection map $O_{G/P}^{-}\times P/\tilde{P}\rightarrow P/\tilde{P}$ takes the form
\begin{center}
$\left[\begin{array}{cc}
A_{n\times n} & 0\\
D_{n\times n} & \JDiag[n](A^{T})^{-1}\JDiag[n]
\end{array}\right]$ mod $\tilde{P}\longmapsto A(\textrm{mod}\,P'_{\widehat{r-k}})\in \GL_n/P'_{\widehat{r-k}}\cong P/\tilde{P}$.
\par\end{center}
\end{asparaenum}
\end{proposition}

\begin{proof}
\noindent \eqref{proposition:structureOMinusSp2nUminus}:Let $U_{\alpha}$ be the root subgroup corresponding to the root $\alpha\in R$. Then the subgroup $U_{P}^{-}$ is generated by $U_{-\alpha}$ for $\alpha\in R^{+}\backslash R_{P}^{+}$. Under the canonical projection $G\rightarrow G/P$ the subgroup $U_{P}^{-}$ is mapped isomorphically onto its image $O_{G/P}^{-}$ (cf. \cite[Section 4.4.4]{BilleyLakshmibaiSingularLoci00}). Thus we obtain the identification of $O_{G/P}^{-}$ with $U_{P}^{-}P/P$.

\noindent \eqref{proposition:structureOMinusSp2ninvertibleA}: Suppose that $zP\in O_{G/P}^{-}$. By ~\eqref{proposition:structureOMinusSp2nUminus} this means
that $\exists$ $n\times n$ matrices $A^{\prime},C^{\prime},D^{\prime},E^{\prime}$
such that

\begin{center}
$z_{1}=\left[\begin{array}{cc}
\Id_{n} & 0\\
D^{\prime} & \Id_{n}
\end{array}\right]\in U_{P}^{-}$ and $z_{2}=\left[\begin{array}{cc}
A^{\prime} & C^{\prime}\\
0 & E^{\prime}
\end{array}\right]\in P$ with $z=\left[\begin{array}{cc}
A & C\\
D & E
\end{array}\right]=z_{1}z_{2}.$
\par\end{center}

\noindent Hence $A=A^{\prime}$, and $A^{\prime}$ invertible
implies $A$ invertible. 

\noindent Conversely, suppose $A$
is invertible. Let 
\begin{center}
$z=\left[\begin{array}{cc}
A & C\\
D & E
\end{array}\right]\in G$.
\end{center}
Then $A,C,D,\, E$ satisfy properties ~\eqref{equation:StructureSP2nAD}-\eqref{equation:StructureSP2nCE}. Since  $A$ is invertible we may write
\begin{center}
$z=z_{1}z_{2}$ where $z_{1}=\left[\begin{array}{cc}
\Id_{n} & 0\\
DA^{-1} & \Id_{n}
\end{array}\right]$, $z_{2}=\left[\begin{array}{cc}
A_{\,} & C\\
0 & E-DA^{-1}C_{\,}
\end{array}\right]$
\par\end{center}

\noindent We shall now show that $z_1,z_2 \in G$.
First, we note that ~\eqref{equation:StructureSP2nAD} implies that
\begin{equation}
\label{equation:identitySP2n1}
\JDiag[n](DA^{-1})=(DA^{-1})^{T}\JDiag[n].
\end{equation}

\noindent Then ~\eqref{equation:identitySP2n1} shows that $z_{1}\in U_{P}^{-}$, and hence $z_{1}\in G$. 

Now $z_{1}\in G$ implies $z_{1}^{-1}\in G$, and $z\in G$ by assumption. Hence $z_{2}=zz_{1}^{-1}\in G$. Further, since $A$ is invertible $z_{2}\in P$. Hence the coset $zP=z_{1}P$, which in view of the fact that $z_{1}\in U_{P}^{-}$, implies by part \eqref{proposition:structureOMinusSp2nUminus} that $zP \in O_{G/P}^{-}$.

\noindent \eqref{proposition:structureOMinusSp2nInverseImage}: Let $z\in U_{P}^{-}P\subset G$. Then we can write
$z=z_{1}z_{2}$ uniquely with $z_{1}\in U_{P}^{-}$, $z_{2}\in P$.
Suppose that

\begin{center}
$\left[\begin{array}{cc}
\Id_{n} & 0\\
D_{n\times n} & \Id_{n}
\end{array}\right]\left[\begin{array}{cc}
A_{n\times n} & C_{n\times n}\\
0_{n\times n} & E_{n\times n}
\end{array}\right]=\left[\begin{array}{cc}
\Id_{n} & 0\\
D_{n\times n}^{\prime} & \Id_{n}
\end{array}\right]\left[\begin{array}{cc}
A_{n\times n}^{\prime} & C_{n\times n}^{\prime}\\
0_{n\times n} & E_{n\times n}^{\prime}
\end{array}\right]$
\par\end{center}

\noindent then $A=A^{\prime}$, $C=C^{\prime}$, $DA=D^{\prime}A^{\prime}$
and $DC+E=D^{\prime}C^{\prime}+E^{\prime}$, which yields that $D^{\prime}=D$
(since $A=A^{\prime}$ is invertible), and then $E=E^{\prime}$. Hence
$U_{P}^{-}\times_{\mathbb{C}}P=U_{P}^{-}P$. Thus for any parabolic
subgroup $P^{\prime}\subseteq P$, $U_{P}^{-}\times_{\mathbb{C}}P/P^{\prime}=U_{P}^{-}P/P^{\prime}$.
The asserted isomorphism follows by part \eqref{proposition:structureOMinusSp2nUminus} from taking $P^{\prime}=\tilde{P}$.

\noindent To see the second assertion consider

\begin{center}
$z=\left[\begin{array}{cc}
A_{n\times n} & C_{n\times n}\\
D_{n\times n} & E_{n\times n}
\end{array}\right]\in G$
\par\end{center}

\noindent with $zP\in O_{G/P}^{-}$. Note that the $n\times n$ block matrices satisfy properties ~\eqref{equation:StructureSP2nAD}-\eqref{equation:StructureSP2nADCE} and by ~\eqref{proposition:structureOMinusSp2ninvertibleA} $A$ is invertible.

We have by the first part of ~\eqref{proposition:structureOMinusSp2nInverseImage} that the coset $zP$ is an element of $O_{G/P}^{-}\times P/\tilde{P}$, since $zP\in O_{G/P}^{-}$. 

\noindent\textbf{Claim} We have a decomposition of $z$ in $G$,

\begin{center}
$\left[\begin{array}{cc}
A_{\,} & C\\
D & E_{\,}
\end{array}\right]=y_{1}y_{2}$ where $y_{1}=\left[\begin{array}{cc}
A_{\,} & 0\\
D & \JDiag[n](A^{T})^{-1}\JDiag[n]
\end{array}\right]\in G$, $y_{2}=\left[\begin{array}{cc}
\Id_{n} & A^{-1}C\\
0 & \Id_{n}
\end{array}\right]\in \tilde{P}$.
\par\end{center}

\noindent We first check that $z=y_{1}y_{2}$. We need the following identity
\begin{equation}
\label{equation:identitySP2n3}
\JDiag[n]A^{T}\JDiag[n](E-DA^{-1}C)=\Id_n
\end{equation}
which follows from 
\begin{center}
$\begin{array}{rlr}
\JDiag[n]A^{T}\JDiag[n](E-DA^{-1}C) & = \JDiag[n](A^{T}\JDiag[n]E-A^{T}\JDiag[n]DA^{-1}C) & \\
\; & =\JDiag[n](A^{T}\JDiag[n]E-D^{T}\JDiag[n]AA^{-1}C) & ~\eqref{equation:StructureSP2nAD} \\
\; & =\JDiag[n]\JDiag[n] & ~\eqref{equation:StructureSP2nADCE}\\
\; & =\Id_n & \;\\
\end{array}$
\end{center}

So that
\begin{center}
$\begin{array}{rlr}
DA^{-1}C+\JDiag[n](A^{T})^{-1}\JDiag[n] & = DA^{-1}C+\JDiag[n](A^{T})^{-1}\JDiag[n]\JDiag[n]A^{T}\JDiag[n](E-DA^{-1}C) & ~\eqref{equation:identitySP2n3}\\
\; & =DA^{-1}C+E-DA^{-1}C & \; \\
\; & =E & \;\\
\end{array}$
\end{center}

With this it is easily verified that $z=y_{1}y_{2}$.\\
It is clear that $y_{1}\in G$. To show $y_{2}\in G$ we need to check that $\JDiag[n](A^{-1}C)^{T}\JDiag[n]=A^{-1}C$.
\begin{center}
$\begin{array}{rlr}
(A^{-1}C)^{T}\JDiag[n] & = (A^{-1}C)^{T}\JDiag[n]\JDiag[n]A^{T}\JDiag[n](E-DA^{-1}C) & ~\eqref{equation:identitySP2n3}\\
\; & =C^{T}\JDiag[n](E-DA^{-1}C) & \; \\
\; & =E^{T}\JDiag[n]C-C^{T}\JDiag[n]DA^{-1}C & ~\eqref{equation:StructureSP2nCE}\\
\; & =(E-DA^{-1}C)^{T}\JDiag[n]C & ~\eqref{equation:identitySP2n1}\\
\; & =(E-DA^{-1}C)^{T}\JDiag[n]A\JDiag[n]\JDiag[n]A^{-1}C & \; \\
\; & =(\JDiag[n]A^{T}\JDiag[n](E-DA^{-1}C))^{T}\JDiag[n](A^{-1}C) & \;\\
\; & =\JDiag[n](A^{-1}C) & ~\eqref{equation:identitySP2n3} \\
\end{array}$
\end{center}

\noindent Thus $y_{2}\in G$. It is clear additionally that
$y_{2}\in\tilde{P}$(in fact $y_{2}\in B_{G}$).

\noindent Hence our claim follows and we have

\begin{center}
$\left[\begin{array}{cc}
A_{\,} & C\\
D & E_{\,}
\end{array}\right]=\left[\begin{array}{cc}
A_{\,} & 0\\
D & \JDiag[n](A^{T})^{-1}\JDiag[n]
\end{array}\right](\textrm{mod}\,\tilde{P})$.
\par\end{center}

\noindent In conclusion,
\[
\begin{bmatrix}
A_{n \times n} & 0_{n \times n} \\
D_{n \times n} & \JDiag[n](A^{T})^{-1}\JDiag[n]
\end{bmatrix} \;=\;
\begin{bmatrix}
A'_{n \times n} & 0_{n \times n} \\
D'_{n \times n} & \JDiag[n](A'^{T})^{-1}\JDiag[n]
\end{bmatrix}(\textrm{mod}\,\tilde{P})
\] 
\noindent if and only if there exist matrices $q \in P'_{\widehat{r-k}}$ and
$L \in \Mat_n$ such that \[
\begin{bmatrix} A' & 0_{n \times n} \\ D' & \JDiag[n](A^{T})^{-1}\JDiag[n] \end{bmatrix} =
\begin{bmatrix} A & 0_{n \times n} \\ D & \JDiag[n](A'^{T})^{-1}\JDiag[n] \end{bmatrix} \begin{bmatrix} q & L \\ 0_{n \times n} & \JDiag[n](q^{T})^{-1}\JDiag[n]
\end{bmatrix}, \] 
which, since both $A$ and $A'$ are invertible, holds if and only if $L=0$, $A' = Aq$ and $D' = Dq$.

\noindent \eqref{proposition:structureOMinusSp2nSLnModQ}: Define a map from $P \rightarrow \GL_n$ by
\begin{center}
$\left[\begin{array}{cc}
A & C\\
0 & E
\end{array}\right]\rightarrow A$.
\par\end{center}
This map is clearly surjective. The fact that $\tilde{P}=P_{\widehat{r-k},\widehat{n}}$ precisely implies that this map induces our desired isomorphism $P/\tilde{P} \cong \GL_n/P'_{\widehat{r-k}}$. Further, the element
\begin{center}
$\left[\begin{array}{cc}
A & C\\
D & E
\end{array}\right] (\textrm{mod}\,\tilde{P}) \in O_{G/P}^{-}\times P/\tilde{P}$
\par\end{center}
has a unique decomposition as
\begin{center}
$\left[\begin{array}{cc}
\Id_{n} & 0\\
DA^{-1} & \Id_{n}
\end{array}\right]
\left(\left[\begin{array}{cc}
A_{\,} & C\\
0 & E-DA^{-1}C_{\,}
\end{array}\right](\textrm{mod}\,\tilde{P})\right)$
\par\end{center}
which implies that this element is mapped to $A(\textrm{mod}\,P'_{\widehat{r-k}})$.

\end{proof}

\subsection{Homogeneous Bundles and Representations}
\label{subsection:Preliminaries:homogVectorBundles}

Any parabolic subgroup $Q$ of $\GL_n$ has a Levi decomposition that expresses the parabolic subgroup as a semidirect product of its Levi part $L_Q$ and its unipotent radical $U_Q$. Let $E$ be a finite dimensional vector space that is a right $Q$-module. Any such $E$ induces a vector bundle over $\GL_N / Q$ as follows:
\begin{center}
$\tilde{E}:= \GL_n {\times}^Q E := \GL_n \times E / \sim$
\end{center}
where $\sim$ is the equivalence relation $(g,e)\sim(gq,eq)$ for all $q\in Q$. Then $\pi_E: \tilde{E} \rightarrow \GL_n / Q$ is the map $\pi_E((g,e)) = gQ$ (it is an easy check to verify that this is well-defined). There is an induced left action of $\GL_n$ on $\tilde{E}$ given by $h\cdot(g,e)=(hg,e)$ for $h,g \in \GL_n$ and $e \in E$. The vector bundle $\tilde{E}$ is homogeneous, that is, the map $\pi_E$ is $\GL_n$-equivariant; trivially $\pi_E(h\cdot (g,e)) = h \cdot \pi_E((g,e)) = hgQ$. In fact, any homogeneous vector bundle on $\GL_n / Q$ arises in this way.

In light of this definition, we will call $\tilde{E}$ irreducible if the associated $Q$-module $E$ is irreducible. In the same way we will say $\tilde{E}$ is indecomposable or completely reducible if $E$ is indecomposable or completely reducible, respectively. We will need to use the fact that $E$ is completely reducible if and only if the unipotent radical $U_Q$ acts trivially (cf. \cite[Section~5]{SnowHomogVB} or \cite[Section~10]{OttavianiRHV95}).

We are interested in using the above construction to compute cohomology groups of vector bundles on the Grassmannian $\GL_N / P_{\widehat{d}}$. Let $\lambda=(\lambda_1,\ldots,\lambda_n)$ be a weight given in the $\epsilon$ basis. Then $\lambda$ is called $P_{\widehat{d}}$-dominant if $\lambda_1 \geq \cdots \geq \lambda_d$ and $\lambda_{d+1} \geq \cdots \geq \lambda_n$. The $P_{\widehat{d}}$-dominant weights are in bijection with the irreducible homogeneous vector bundles on $\GL_N / P_{\widehat{d}}$. 

Given a partition $\lambda$, denote the associated Schur functor by $\Schur_{\lambda}$ (cf. \cite[\S 6.1]{FulHarRepTheory91}). Let $\mathcal{R}$ and $\mathcal{Q}$ be the tautological sub-bundle and tautological quotient bundle, respectively, on the Grassmannian $\GL_N / P_{\widehat{d}}$. For a $P_{\widehat{d}}$-dominant weight $\lambda=(\lambda_1,\ldots,\lambda_n)$ the corresponding irreducible homogeneous vector bundle is $V(\lambda):=\Schur_{(\lambda_1,\ldots,\lambda_d)} \mathcal{Q}^* \otimes \Schur_{(\lambda_{d+1},\ldots,\lambda_n)} \mathcal{R}^*$ (cf. ~\cite[\S 4]{WeymanCohVBSyz03}). Note that the construction in ~\cite{WeymanCohVBSyz03} uses left modules which results in the irreducible vector bundle associated $\lambda$ being $\Schur_{(\lambda_1,\ldots,\lambda_d)} \mathcal{R}^* \otimes \Schur_{(\lambda_{d+1},\ldots,\lambda_n)} \mathcal{Q}^*$; in our case we are working with right modules.

\subsection{Bott algorithm} We now give a brief description of the Bott-algorithm for computing the cohomology of irreducible homogeneous vector bundles on $\GL_n/Q$~\cite[Remark 4.1.5]{WeymanCohVBSyz03}. 

Let $\alpha=(\alpha_1,\ldots,\alpha_n)$ be a weight. As in \cite[Remark 4.1.5]{WeymanCohVBSyz03} we define an action of the permutation $\nu_i=(i, i+1)$ on the set of weights in the following way:
\begin{equation}
\label{equation:exchangeAction}
\nu_i \alpha = ((\alpha_1,\ldots,\alpha_{i-1},\alpha_{i+1} - 1, \alpha_{i} + 1,\alpha_{i+2},\ldots,\alpha_n).
\end{equation}

The Bott-algorithm may be applied to our case as follows. For $P_{\widehat{d}}$,
with $1 \leq d \leq n-1$ and let $\lambda=(\lambda_1,\ldots,\lambda_n)$ be a $P_{\widehat{d}}$-dominant weight with associated homogeneous vector bundle $V(\lambda):=\Schur_{(\lambda_1,\ldots,\lambda_d)} \mathcal{Q}^* \otimes \Schur_{(\lambda_{d+1},\ldots,\lambda_n)} \mathcal{R}^*$. We will apply the Bott-algorithm to $\lambda'=(\lambda_{d+1},\ldots,\lambda_n,\lambda_1,\ldots,\lambda_d)$.

If $\lambda'$ is nonincreasing, then $\homology^{0}(\GL_n/P_{\widehat{d}},V(\lambda))=\Schur_{\lambda'}\mathbb{C}^{n}$ and $\homology^{i}(\GL_n/P_{\widehat{d}},V(\lambda))=0$ for $i>0$. Otherwise we start to apply the exchanges of type \eqref{equation:exchangeAction} to $\lambda'$, trying to move smaller numbers on the left to the right. Two possibilities can occur:
\begin{enumerate}[label=(\arabic*)]
\item We apply an exchange of type \eqref{equation:exchangeAction} and it leaves the sequence unchanged. In this case  $\homology^{i}(\GL_n/P_{\widehat{d}},V(\lambda))=0$ for $i\geq0$.
\item After applying j exchanges, we transform $\lambda'$ into a nonincreasing
sequence $\beta$. Then we have $\homology^{i}(\GL_n/P_{\widehat{d}},V(\lambda))=0$ for $i\neq j$ and $\homology^{j}(\GL_n/P_{\widehat{d}},V(\lambda))=\Schur_{\beta}\mathbb{C}^{n}$.
\end{enumerate}

\section{Properties of Schubert Desingularization in Type C}
\label{sec:desing}
\numberwithin{equation}{section}

In this section we prove the results necessary to construct the commutative diagram required by the geometric technique. We will need to make use of the following result.

\begin{proposition}
\label{proposition:typeATangentSpace}
Let $Q$ be a parabolic subgroup of $\SL_{2n}$. Let $\tau\in W^{Q}$. Then the dimension of the tangent space of $X_{Q}(\tau)$ at $e_{id}$ is $\#\{s_{\alpha}|\alpha\in R^{-}\backslash R_{Q}^{-}\,\, and\,\,\tau\geq s_{\alpha}\,\, in\,\, W/W_{Q}\}$.
In particular, $X_{Q}(\tau)$ is smooth if and only if dim$X_{Q}(\tau)=\#\{s_{\alpha}|\alpha\in R^{-}\backslash R_{Q}^{-}\,\, and\,\,\tau\geq s_{\alpha}\,\, in\,\, W/W_{Q}\}$.
\end{proposition}
\noindent See \cite[Chapter 4]{BilleyLakshmibaiSingularLoci00} for a proof.

\begin{notation}
\label{notation:parabolicSubgroupsStr}For an integer $i$ with $1\leq i\leq n$ we define $i^{\prime}=2n+1-i$. Let $1\leq k<r\leq n$. Then 
\begin{center}
$\mathcal{W}_{k,r}=
\begin{cases}
(k+1,...,r,n^{\prime},...,(r+1)^{\prime},k^{\prime},...,1^{\prime})\in W^{P}&\text{if $r<n$}\\
(k+1,...,r,k^{\prime},...,1^{\prime})\in W^{P}&\text{if $r=n$}\\
\end{cases}$
\end{center}
Let $1\leq k<r\leq n$ be integers. Let $w = \mathcal{W}_{k,r}$ with $\tilde{w}$ its minimal representative in $W^{\tilde{P}}$. 
\end{notation}

\begin{proposition} 
\label{proposition:schubertVarietySL2nSmooth}
The Schubert variety $X_{\tilde{Q}}(\tilde{w})$ in $H/\tilde{Q}$ is smooth.
\end{proposition}

\begin{proof}
Let $w_{max}\in W_{H}(=S_{2n})$ be the
maximal representative of $\tilde{w}$. Then we claim that
\begin{center}
$w_{max}=
\begin{cases}
([r,k+1][1^{\prime},k^{\prime}][(r+1)^{\prime},n^{\prime}][n,(r+1)][k,1][(k+1)^{\prime},r^{\prime}])&\text{if $r<n$}\\
([r,k+1][1^{\prime},k^{\prime}][k,1][(k+1)^{\prime},r^{\prime}])&\text{if $r=n$}\\
\end{cases}$
\end{center} 
To verify our claim we must show that $w_{max}$ is the maximal element of $W_H$ such that $X_{P_{\hat{i}}}(w_{max})=X_{P_{\hat{i}}}(\tilde{w})$ for $i=r-k,n,2n-(r-k)$. This is immediate from the fact that if $\tau=(a_{1},...,a_{2n})\in W_{H}$ and $\tau^{\prime}\in W^{P_{\hat{i}}}$ is the sequence with $a_1,\ldots,a_i$ in increasing order, then $X_{P_{\hat{i}}}(\tau)=X_{P_{\hat{i}}}(\tau^{\prime})$ for $1\leq i\leq2n$.

Thus $X_{B_{H}}(w_{max})$ is the inverse image of $X_{\tilde{Q}}(\tilde{w})$
under the natural morphism $H/B_{H}\rightarrow H/\tilde{Q}$. As $w_{max}$
is a $4231$ and $3142$ avoiding element of $W_{H}$ we have that
$X_{B_{H}}(w_{max})$ is nonsingular (cf. \cite[\S 8.1.1]{BilleyLakshmibaiSingularLoci00}). Since
the morphism $H/B_{H}\rightarrow H/\tilde{Q}$ has nonsingular fibers
(namely $\tilde{Q}/B_{H}$) we have $X_{\tilde{Q}}(\tilde{w})$ must
be smooth.
\end{proof}

\begin{proposition} 
\label{proposition:schubertVarietySP2nSmooth}
The Schubert variety $X_{\tilde{P}}(\tilde{w})$ in $G/\tilde{P}$ is smooth.
\end{proposition}

\begin{proof} Let $w_{max}$ be as in the proof of Proposition \ref{proposition:schubertVarietySL2nSmooth}. Then $w_{max}$ is in $W_G$ and $X_{B_{G}}(w_{max})$ is the inverse image of $X_{\tilde{P}}(\tilde{w})$ under the natural map $G/B_{G}\rightarrow G/\tilde{P}$. We claim $X_{B_{G}}(w_{max})$ is smooth.

Note that the claim implies the required result (since, the canonical morphism $G/B_G\rightarrow G/\tilde{P}$ is a fibration with nonsingular fibers (namely, $\tilde{P}/B_G$)). To prove the claim, as seen in the proof of Proposition \ref{proposition:schubertVarietySL2nSmooth}, we have that  $X_{B_{H}}(w_{max})$ is smooth. We conclude the smoothness of $X_{B_{G}}(w_{max})$ using the following two formulas \cite[\S 3(VI), Remark 5.8]{LakshmibaiG/PVII87}: Let $\theta\in W_G$, say, $\theta=(a_1,\cdots a_n)$. 
$$l_G(\theta)=\frac{1}{2}[l_H(\theta)+m(\theta)]\leqno{(1)}$$ 
recall (cf. \eqref{equation:mw}) that $m(\theta)=\#\{i,1\le i\le m|a_i>m\}$. 
$$dim\ T_{id}(\theta,G)=\frac{1}{2}[dim\ T_{id}(\theta,H)+c(\theta)] \leqno{(2)}$$
 where $c(\theta)=\#\{1\le i\le m\,|\,\theta\ge s_{\epsilon_{2i}}\}$, and $T_{id}(\theta,G)$(resp $T_{id}(\theta,H)$) denotes the Zariski tangent space of $X_{B_G}(\theta)$(resp $X_{B_H}(\theta)$) at $e_{id}$. Note that $s_{\epsilon_{2i}}$ is just the transposition $(i,i')$(cf. ~\ref{equation:simpleReflectionsSP2n}). Now taking $\theta=w_{max}$, 
we have, $c(w_{max})=m(w_{max})$. Hence we obtain from (1), (2) that dim$\, T_{id}(w_{max},G)=l_G(w_{max})$, proving that $X_{B_{G}}(w_{max})$ is smooth at $e_{id}$, and hence is nonsingular (note that for a Schubert variety X, the singular locus of $X$, $Sing(X)$, is $B$-stable implying $e_{id}\in Sing(X)$ if $Sing(X)\neq \emptyset$). Thus the claim  (and hence the required result) follows.
\end{proof}
\begin{remark}
We have that $X_{\tilde{P}}(\tilde{w})$  is the fixed
point set under an automorphism of  order two of the Schubert variety $X_{\tilde{Q}}(\tilde{w})$
and thus is smooth, provided char$\,K\ne 2$.\cite[Proposition 3.4]{EdixNeron92}
\end{remark}

\begin{discussionbox}
\label{discussionbox:OMinusPlucker}
To give a characterization of $Y_{\tilde{Q}}(\tilde{w})$ we first need a review of the structure of $O^{-}_{H/\tilde{Q}}$ and its Pl\"ucker coordinates.

Recall that for the Pl\"ucker embedding of the Grassmannian $\Grass_{d,n}$, the \emph{Pl\"ucker coordinate} $p_{\underline{i}}(U)$, $U\in \Grass_{d,n}$ and $\underline{i}=(i_1,\ldots,i_d)$ with $1\leq i_1 < \ldots < i_d < n$, is just the $d\times d$ minor of the matrix $A_{n\times d}$ with row indices $(i_1,\ldots,i_d)$ (here the matrix $A_{n\times d}$ represents the $d$-dimensional subspace $U$ with respect to the standard basis).

The space $O^{-}_{H/\tilde{Q}}$ can be identified with the affine space of lower-triangular matrices with possible
non-zero entries $x_{ij}$ at row $i$ and column $j$ where $(i,j)$ is such that there exists a  $l \in \{r-k, n, 2n-(r-k)\}$ such that $j \leq l < i \leq N$. To see this, note that we are interested in those $(i,j)$ such that the root $\epsilon_i - \epsilon_j$ belongs to  $R^{-} \setminus R_{\tilde{Q}}^{-}$. Since $R_{\tilde{Q}}^{-} = R_{Q_{\widehat{r-k}}}^{-} \bigcap R_{Q_{\widehat{n}}} \bigcap R_{Q_{\widehat{2n-(r-k)}}}$, we see that we are looking for $(i,j)$ such that $\epsilon_i - \epsilon_j \in R^{-} \setminus R_{Q_{\widehat{l}}}^{-}$, for some $l \in \{r-k, n, 2n-(r-k)\}$. For the maximal parabolic subgroup $P_{\widehat{l}}\,$, we have, $R^{-} \setminus R_{Q_{\widehat{l}}}^{-} = \{\epsilon_i-\epsilon_j \mid 1\le j\le l<i\le N\}$. We have $\dim O^{-}_{H/\tilde{Q}} = |R^{-}\setminus R^{-}_{\tilde{Q}}|$.

Thus we have the following identification 
\begin{equation}
\label{equation:StructureOMinusTQ}
O^{-}_{H/\tilde{Q}} = \left[\begin{array}{cccc}
\Id_{r-k} & 0 & 0 & 0 \\
A' & \Id_{n-(r-k)} & 0 & 0 \\
\mathcal{D}_{1} & \mathcal{D}_{2} & \Id_{n-(r-k)} & 0 \\
\mathcal{D}_{3} & \mathcal{D}_{4} & E' & \Id_{r-k} \\
\end{array}\right]
\end{equation}
\\ \noindent where the block matrices have possible non-zero entries $x_{ij}$ given by \\
\makebox[0.47\textwidth][l]{$A'=\left[\begin{array}{ccc} x_{(r-k)+1\;\;1}&\ldots&x_{(r-k)+1\;\;r-k} \\
\vdots&\;&\vdots \\
x_{n\;\;1}&\ldots&x_{n\;\;r-k} \\
\end{array}\right]$}
\makebox[0.5\textwidth][l]{$E'=\left[\begin{array}{ccc} x_{2n-(r-k)+1\;\;n+1} & \ldots & x_{2n-(r-k)+1\;\;2n-(r-k)} \\
\vdotswithin{\ldots} & \; & \vdotswithin{\ldots} \\
x_{2n\;\;n+1} & \ldots & x_{2n\;\;2n-(r-k)} \\
\end{array}\right]$} \par \noindent
\makebox[0.47\textwidth][l]{$\mathcal{D}_{1}=\left[\begin{array}{ccc}
x_{n+1\;\;1} & \ldots & x_{n+1\;\;r-k} \\
\vdotswithin{\ldots} & \; & \vdotswithin{\ldots} \\
x_{2n-(r-k)\;\;1} & \ldots & x_{2n-(r-k)\;\;r-k} \\
\end{array}\right]$}
\makebox[0.5\textwidth][l]{$\mathcal{D}_{2}=\left[\begin{array}{ccc}
x_{n+1\;\;(r-k)+1} & \ldots & x_{n+1\;\;n} \\
\vdotswithin{\ldots} & \; & \vdotswithin{\ldots} \\
x_{2n-(r-k)\;\;(r-k)+1} & \ldots & x_{2n-(r-k)\;\;n} \\
\end{array}\right]$} \par \noindent
\makebox[0.47\textwidth][l]{$\mathcal{D}_{3}=\left[\begin{array}{ccc}
x_{2n-(r-k)+1\;\;1} & \ldots & x_{2n-(r-k)+1\;\;r-k} \\
\vdotswithin{\ldots} & \; & \vdotswithin{\ldots} \\
x_{2n\;\;1} & \ldots & x_{2n\;\;r-k} \\
\end{array}\right]$}
\makebox[0.5\textwidth][l]{$\mathcal{D}_{4}=\left[\begin{array}{ccc}
x_{2n-(r-k)+1\;\;(r-k)+1} & \ldots & x_{2n-(r-k)+1\;\;n} \\
\vdotswithin{\ldots} & \; & \vdotswithin{\ldots} \\
x_{2n\;\;(r-k)+1} & \ldots & x_{2n\;\;n} \\
\end{array}\right]$} \par \noindent

We may break the Pl\"ucker coordinates we want to understand into several cases.

\textbf{Case 1:} For $i>r$, $j\leq r-k$ the Pl\"ucker coordinate $p_{(i,j)}^{(r-k)}$ on the Grassmannian $H/Q_{\widehat{r-k}}$ lifts to a regular function on $H/\tilde{Q}$. Its restriction to $O_{H/\tilde{Q}}^{-}$ is the $r-k\times r-k$ minor of ~\eqref{equation:StructureOMinusTQ} with column indices $\{1,2,\ldots,r-k\}$ and row indices $\{1,\ldots,j-1,j+1,\ldots,r-k,i\}$. This minor is the determinant of a $r-k\times r-k$ matrix with the top $(r-k)-1$ rows equal to $\Id_{r-k}$ omitting the $j$th row, and the bottom row equal to the first $r-k$ entries of the $i$th row of ~\eqref{equation:StructureOMinusTQ}. The determinant of this matrix is thus $(-1)^{(r-k)-j}x_{ij}$. Thus for $i>r$, $j\leq r-k$ \begin{equation}
\label{equation:PluckerRMK}
p_{(i,j)}^{(r-k)}\big{|}_{O_{H/\tilde{Q}}^{-}}=(-1)^{(r-k)-j}x_{ij}
\end{equation} 

\textbf{Case 2:} For $i>2n-(r-k)$, $n<j\leq 2n-(r-k)$ the Pl\"ucker coordinate $p_{(i,j)}^{(2n-(r-k))}$ on the Grassmannian $H/Q_{\widehat{2n-(r-k)}}$ lifts to a regular function on $H/\tilde{Q}$. Its restriction to $O_{H/\tilde{Q}}^{-}$ is the $2n-(r-k)\times 2n-(r-k)$ minor of ~\eqref{equation:StructureOMinusTQ} with column indices $\{1,2,\ldots,2n-(r-k)\}$ and row indices $\{1,\ldots,j-1,j+1,\ldots,2n-(r-k),i\}$. This minor is the determinant of \begin{equation}
\label{equation:Plucker2nRMKMatrix}
\left[\begin{array}{ccc}
\Id_{r-k} & 0 & 0\\
A' & \Id_{n-(r-k)} & 0 \\
\hat{\mathcal{D}_{1}} & \hat{\mathcal{D}_{2}} & \hat{I_1}\\
\left[x_{i\;1}\ldots x_{i\;r-k}\right] & \left[x_{i\;(r-k)+1}\ldots x_{i\;n}\right] & \left[x_{i\;n+1}\ldots x_{i\;2n-(r-k)}\right]
\end{array}\right]
\end{equation} where $\hat{\mathcal{D}_{1}}, \hat{\mathcal{D}_{2}},$ and $\hat{I_1}$ are equal to, respectively,  $\mathcal{D}_{1}, \mathcal{D}_{2},$ and $\Id_{n-(r-k)}$ with their $(j-n)$th rows omitted. The determinant of ~\eqref{equation:Plucker2nRMKMatrix} is equal to the determinant of 
\begin{center}
$\left[\begin{array}{c}
\hat{I_1} \\
\left[x_{i\;n+1}\ldots x_{i\;2n-(r-k)}\right] \\
\end{array}\right]$
\end{center}
As above this is just an identity matrix with a single row replaced and so its determinant is just equal $(-1)^{2n-(r-k)-j}x_{ij}$. Thus for $i>2n-(r-k)$, $n<j\leq 2n-(r-k)$\begin{equation}
\label{equation:Plucker2nRMK}
p_{(i,j)}^{(2n-(r-k))}\big{|}_{O_{H/\tilde{Q}}^{-}}=(-1)^{2n-(r-k)-j}x_{ij}
\end{equation} 

\textbf{Case 3:} For $i>2n-(r-k)$, $r-k<j\leq n$ the Pl\"ucker coordinate $p_{(i,j)}^{(2n-(r-k))}$ on the Grassmannian $H/Q_{\widehat{2n-(r-k)}}$ lifts to a regular function on $H/\tilde{Q}$. Its restriction to $O_{H/\tilde{Q}}^{-}$ is the $2n-(r-k)\times 2n-(r-k)$ minor of ~\eqref{equation:StructureOMinusTQ} with column indices $\{1,2,\ldots,2n-(r-k)\}$ and row indices $\{1,\ldots,j-1,j+1,\ldots,2n-(r-k),i\}$. This minor is the determinant of\begin{equation}
\label{equation:Plucker2nRMKMatrix+}
\left[\begin{array}{ccc}
\Id_{r-k} & 0 & 0\\
\hat{A'} & \hat{I_2} & 0 \\
\mathcal{D}_{1} & \mathcal{D}_{2} & \Id_{n-(r-k)}\\
\left[x_{i\;1}\ldots x_{i\;r-k}\right] & \left[x_{i\;(r-k)+1}\ldots x_{i\;n}\right] & \left[x_{i\;n+1}\ldots x_{i\;2n-(r-k)}\right]
\end{array}\right]
\end{equation} where $\hat{A'}$ and $\hat{I_2}$ are equal to, respectively,  $A'$ and $\Id_{n-(r-k)}$ with their $j-(r-k)$th rows omitted. The determinant of ~\eqref{equation:Plucker2nRMKMatrix+} is equal to the determinant of \begin{equation}
\label{equation:Plucker2nRMKMatrix++}
\left[\begin{array}{cc}
\hat{I_2} & 0 \\
\mathcal{D}_{2} & \Id_{n-(r-k)}\\
\left[x_{i\;(r-k)+1}\ldots x_{i\;n}\right] & \left[x_{i\;n+1}\ldots x_{i\;2n-(r-k)}\right]
\end{array}\right]
\end{equation} To calculate this shift the bottom row so that it becomes the $j-(r-k)$th row of $\hat{I_2}$. Let $M=2n-(r-k)-j$. Then the determinant of ~\eqref{equation:Plucker2nRMKMatrix++} will be $(-1)^{M}$ times the determinant of\begin{equation}
\label{equation:Plucker2nRMKMatrix+++}
\left[\begin{array}{cc}
I_3 & Z \\
\mathcal{D}_{2} & \Id_{n-(r-k)}\\
\end{array}\right]
\end{equation} where $I_3$ is $\Id_{n-(r-k)}$ with the $j-(r-k)$th row replaced by $\left[x_{i\;(r-k)+1}\ldots x_{i\;n}\right]$ and $Z$ is the zero matrix with the $j-(r-k)$th row replaced by $\left[x_{i\;n+1}\ldots x_{i\;2n-(r-k)}\right]$. Since the lower right block matrix of ~\eqref{equation:Plucker2nRMKMatrix+++} commutes with its lower left block matrix we have that the determinant of ~\eqref{equation:Plucker2nRMKMatrix+++} is equal to the determinant of $I_3-Z\mathcal{D}_{2}$.
We have that $Z\mathcal{D}_{2}$ is equal to the zero matrix with its $j-(r-k)$th row replaced by 
\begin{center}
$\left[x_{i\;(r-k)+1}\ldots x_{i\;n}\right]\mathcal{D}_{2}$
\end{center}
And thus $I_3-Z\mathcal{D}_{2}$ is equal to $\Id_{n-(r-k)}$ with the $j-(r-k)$th row replaced by 
\begin{center}
$\left[x_{i\;(r-k)+1}\ldots x_{i\;n}\right] - \left[x_{i\;(r-k)+1}\ldots x_{i\;n}\right]\mathcal{D}_{2}$
\end{center}
And so the determinant of $I_3-Z\mathcal{D}_{2}$ is merely equal to the $j-(r-k)$th entry of $I_3-Z\mathcal{D}_{2}$ which is
\begin{center}
$x_{ij} - \left[x_{i\;(r-k)+1}\ldots x_{i\;n}\right]\left[x_{n+1\;j}\ldots x_{2n-(r-k)\;j}\right]^{T}$
\end{center}
Combining all our steps, we finally have that for $i>2n-(r-k)$, $r-k<j\leq n$ 
\begin{equation}
\label{equation:Plucker2nRMK+}
p_{(i,j)}^{(2n-(r-k))}\big{|}_{O_{H/\tilde{Q}}^{-}}=(-1)^{M}(x_{ij} - \left[x_{i\;(r-k)+1}\ldots x_{i\;n}\right]\left[x_{n+1\;j}\ldots x_{2n-(r-k)\;j}\right]^{T})
\end{equation}
\vspace{-10pt}\end{discussionbox}

\begin{theorem} 
\label{theorem:schubertVarietySL2nMatrixForm}
The opposite cell $Y_{\tilde{Q}}(\tilde{w})$ can be identified with the subspace of $O_{H/\tilde{Q}}^{-}$ given by matrices of the form
\begin{center}
$\left[\begin{array}{cccc}
\Id_{r-k} & 0 & 0 & 0 \\
A' & \Id_{n-(r-k)} & 0 & 0 \\
0 & \mathcal{D}_{2} & \Id_{n-(r-k)} & 0 \\
0 & E'\mathcal{D}_{2} & E' & \Id_{r-k} \\
\end{array}\right]$
\end{center}
with $\mathcal{D}_{2} \in  \Mat_{n-(r-k)}$, $A'\in \Mat_{n-(r-k)\times r-k}$ with the bottom $n-r$ rows of $A'$ all zero, and $E'\in \Mat_{r-k\times n-(r-k)}$ with the left $n-r$ columns of $E'$ all zero.

\end{theorem}

\begin{proof}
For $j\leq r-k<i$ the reflection $(i,j)$
equals $(1,2,..,j-1,j+1,..,r-k,i)$ and $\tilde{w}$ equals $(k+1,...,r)$ in $W/W_{Q_{\widehat{r-k}}}$. Thus for $i>r$, and $j\leq r-k$
we have the reflection $(i,j)$ is not smaller than $\tilde{w}$ in
$W/W_{Q_{\widehat{r-k}}}$ so the Pl\"ucker coordinate $p_{(i,j)}^{(r-k)}$
vanishes on $X_{\tilde{Q}}(\tilde{w})$. We saw in ~\eqref{equation:PluckerRMK} that for such $(i,j)$ we have $p_{(i,j)}^{(r-k)}=(-1)^{(r-k)-j}x_{ij}$ and thus $x_{ij}\equiv0$
on $Y_{\tilde{Q}}(\tilde{w})$.

For $j\leq n<i$ the reflection $(i,j)$ equals $(1,2,..,j-1,j+1,..,n,i)$
and $\tilde{w}$ is equal to $(k+1,...,r,n^{\prime},...,(r+1)^{\prime},k^{\prime},...,1^{\prime})$ in $W/W_{Q_{\hat{n}}}$.
Thus there is no choice of $(i,j)$ such that $(i,j)$ is not smaller
than $\tilde{w}$ in $W/W_{Q_{\widehat{n}}}$.

For $j\leq 2n-(r-k)<i$ the reflection $(i,j)$ equals $(1,2,..,j-1,j+1,..,2n-(r-k),i)$
and $\tilde{w}$ equals $(1,...,n,n^{\prime},...,(r+1)^{\prime},k^{\prime},...,1^{\prime})$ in $W/W_{Q_{\hat{2n-(r-k)}}}$ .
Thus for $i>2n-(r-k)$, and $j\leq 2n-r$
we have the reflection $(i,j)$ is not smaller than $\tilde{w}$ in
$W/W_{Q_{\widehat{2n-(r-k)}}}$. We break these into two cases, ignoring those $j\leq r-k$ as we have already shown above that for $j\leq r-k$ and $i>2n-(r-k)$ we have $x_{ij}\equiv0$ on $Y_{\tilde{Q}}(\tilde{w})$.

The first case is for $(i,j)$ with $i>2n-(r-k)$, and $n<j\leq 2n-r$. The fact that $(i,j)$ is not smaller than $\tilde{w}$ in
$W/W_{Q_{\widehat{2n-(r-k)}}}$ implies that the Pl\"ucker coordinate $p_{(i,j)}^{(2n-(r-k))}$ vanishes on $X_{\tilde{Q}}(\tilde{w})$. We saw in ~\eqref{equation:Plucker2nRMK} that for such $(i,j)$ we have $p_{(i,j)}^{(2n-(r-k))}=(-1)^{2n-(r-k)-j}x_{ij}$ and thus $x_{ij}\equiv0$
on $Y_{\tilde{Q}}(\tilde{w})$.

The second case is for $(i,j)$ with $i>2n-(r-k)$ and $r-k<j\leq n$. $(i,j)$ is not smaller than $\tilde{w}$ in
$W/W_{Q_{\widehat{2n-(r-k)}}}$ implies that the Pl\"ucker coordinate $p_{(i,j)}^{(2n-(r-k))}$ vanishes on $X_{\tilde{Q}}(\tilde{w})$. We saw in ~\eqref{equation:Plucker2nRMK+} that for such $(i,j)$ we have $p_{(i,j)}^{(2n-(r-k))}=(-1)^{M}(x_{ij}-\left[x_{i\;(r-k)+1}\ldots x_{i\;n}\right]\left[x_{n+1\;j}\ldots x_{2n-(r-k)\;j}\right]^{T})$. Combining these two facts we get $x_{ij}=\left[x_{i\;(r-k)+1}\ldots x_{i\;n}\right]\left[x_{n+1\;j}\ldots x_{2n-(r-k)\;j}\right]^{T}$. 

Noting that $\left[x_{i\;(r-k)+1}\ldots x_{i\;n}\right]$ is the $(2n-(r-k)-i)$th row of $E'$ and $\left[x_{n+1\;j}\ldots x_{2n-(r-k)\;j}\right]^{T}$ is the $(2n-(r-k)-j)$th column of $\mathcal{D}_{2}$ it is clear that $x_{ij}=(E'X)_{(2n-(r-k)-i)\;(2n-(r-k)-j)}$ on $Y_{\tilde{Q}}(\tilde{w})$.

But the reflections $(i,j)$ with $i>r$ and $j\leq r-k$, and $i>2n-(r-k)$ and $r-k<j\leq 2n-r$ are exactly the reflections $s_{\alpha}$ with $\alpha\in R^{-}\backslash R_{\tilde{Q}}^{-}$ and $\tilde{w}\ngeq s_{\alpha}$ in $W/W_{\tilde{Q}}$. Since $X_{\tilde{Q}}(\tilde{w})$ is smooth this implies by Proposition \ref{proposition:typeATangentSpace} that the codimension of $Y_{\tilde{Q}}(\tilde{w})$ in $O_{H/\tilde{Q}}^{-}$ equals $\#\{(i,j)|\, i>r\,\, and\,\, j\leq r-k,\,\,or\,\,i>2n-(r-k)\,\,and\,\,r-k<j\leq 2n-r\}$. Above we have shown that for each such $(i,j)$, $x_{ij}$ either vanishes, or is completely dependent on the entries of $E'X$. Thus $Y_{\tilde{Q}}(\tilde{w})$ is the subspace of $O_{H/\tilde{Q}}^{-}$ defined by the vanishing of $\{x_{ij}|\, i>r\,\, and\,\, j\leq r-k\,\,,or\,\,i>2n-(r-k)\,\,and\,\,n<j\leq 2n-r\}$ and $x_{ij}=(E'X)_{(2n-(r-k)-i)\;(2n-(r-k)-j)}$ for $i>2n-(r-k)$ and $r-k<j\leq n$.
\end{proof}

\begin{example}
Let $k=2$,$r=4$, and $n=5$. Then $\tilde{Q}=Q_{\hat{2},\hat{5},\hat{8}}$, $w=(3,4,6,9,10)$, and $\tilde{w}=(3,4,6,9,10,1,2,5)$. Then
\begin{center}
$O_{H/\tilde{Q}}^{-}=\left[\begin{array}{cccccccccc}
1 & 0 & 0 & 0 & 0 & 0 & 0 & 0 & 0 & 0\\
0 & 1 & 0 & 0 & 0 & 0 & 0 & 0 & 0 & 0\\
x_{31} & x_{32} & 1 & 0 & 0 & 0 & 0 & 0 & 0 & 0\\
x_{41} & x_{42} & 0 & 1 & 0 & 0 & 0 & 0 & 0 & 0\\
x_{51} & x_{52} & 0 & 0 & 1 & 0 & 0 & 0 & 0 & 0\\
x_{61} & x_{62} & x_{63} & x_{64} & x_{65} & 1 & 0 & 0 & 0 & 0\\
x_{71} & x_{72} & x_{73} & x_{74} & x_{75} & 0 & 1 & 0 & 0 & 0\\
x_{81} & x_{82} & x_{83} & x_{84} & x_{85} & 0 & 0 & 1 & 0 & 0\\
x_{91} & x_{92} & x_{93} & x_{94} & x_{95} & x_{96} & x_{97} & x_{98} & 1 & 0\\
x_{101} & x_{102} & x_{103} & x_{104} & x_{105} & x_{106} & x_{107} & x_{108} & 0 & 1\\
\end{array}\right]$
\end{center}
And $Y_{\tilde{P}}(\tilde{w})$ will be the subspace of $O_{H/\tilde{Q}}^{-}$ given by
\begin{center}
$\left[\begin{array}{cccccccccc}
1 & 0 & 0 & 0 & 0 & 0 & 0 & 0 & 0 & 0\\
0 & 1 & 0 & 0 & 0 & 0 & 0 & 0 & 0 & 0\\
x_{31} & x_{32} & 1 & 0 & 0 & 0 & 0 & 0 & 0 & 0\\
x_{41} & x_{42} & 0 & 1 & 0 & 0 & 0 & 0 & 0 & 0\\
0 & 0 & 0 & 0 & 1 & 0 & 0 & 0 & 0 & 0\\
0 & 0 & x_{63} & x_{64} & x_{65} & 1 & 0 & 0 & 0 & 0\\
0 & 0 & x_{73} & x_{74} & x_{75} & 0 & 1 & 0 & 0 & 0\\
0 & 0 & x_{83} & x_{84} & x_{85} & 0 & 0 & 1 & 0 & 0\\
0 & 0 & x_{97}x_{73}+x_{98}x_{83} & x_{97}x_{74}+x_{98}x_{84} & x_{97}x_{75}+x_{98}x_{85} & 0 & x_{97} & x_{98} & 1 & 0\\
0 & 0 & x_{107}x_{73}+x_{108}x_{83} & x_{107}x_{74}+x_{108}x_{84} & x_{107}x_{75}+x_{108} & 0 & x_{107} & x_{108} & 0 & 1\\
\end{array}\right]$
\end{center}
\end{example}

\begin{corollary} 
\label{corollary:schubertVarietySP2nMatrixForm}
The opposite cell $Y_{\tilde{P}}(\tilde{w})$ can be identified with the subspace of $O_{G/\tilde{P}}^{-}$ given by matrices of the form
\begin{center}
$\left[\begin{array}{cccc}
\Id_{r-k} & 0 & 0 & 0 \\
A' & \Id_{n-(r-k)} & 0 & 0 \\
0 & \mathcal{D}_{2} & \Id_{n-(r-k)} & 0 \\
0 & -\JDiag[r-k](A')^{T}\JDiag[r-k]\mathcal{D}_{2} & -\JDiag[r-k](A')^{T}\JDiag[r-k] & \Id_{r-k} \\
\end{array}\right]$
\end{center}
with $\JDiag[n-(r-k)]\mathcal{D}_{2} \in  \Sym_{n-(r-k)}$ and $A'\in \Mat_{n-(r-k)\times r-k}$ with the bottom $n-r$ rows of $A'$ all zero.

\end{corollary}
\begin{proof}
Let $y \in Y_{\tilde{P}}(\tilde{w})=(Y_{\tilde{Q}}(\tilde{w}))^{\sigma}\subset Y_{\tilde{Q}}(\tilde{w})$. So $y$ is just an element of $Y_{\tilde{Q}}(\tilde{w})$ that is fixed under the involution $\sigma$. That is, an element which satisfies ~\eqref{equation:StructureSP2nAD}-\eqref{equation:StructureSP2nADCE}. Theorem \ref{theorem:schubertVarietySL2nMatrixForm} gives us that $y$ is of the form
\begin{center}
$\left[\begin{array}{cccc}
\Id_{r-k} & 0 & 0 & 0 \\
A' & \Id_{n-(r-k)} & 0 & 0 \\
0 & \mathcal{D}_{2} & \Id_{n-(r-k)} & 0 \\
0 & E'\mathcal{D}_{2} & E' & \Id_{r-k} \\
\end{array}\right]$
\end{center}
with $\mathcal{D}_{2} \in  \Mat_{n-(r-k)}$, $A'\in \Mat_{n-(r-k)\times r-k}$ with the bottom $n-r$ rows of $A'$ all zero, and $E'\in \Mat_{r-k\times n-(r-k)}$ with the left $n-r$ columns of $E'$ all zero.
We must now check what restrictions on $y$ are required for it to satisfy ~\eqref{equation:StructureSP2nAD}-\eqref{equation:StructureSP2nADCE}.
For $y$ to satisfy \eqref{equation:StructureSP2nADCE} we know that 
\begin{center}
$\left[\begin{array}{cc}
\Id_{r-k} & 0 \\
A' & \Id_{n-(r-k)} \\
\end{array}\right]^{T}
\left[\begin{array}{cc}
0 & \JDiag[r-k] \\
\JDiag[n-(r-k)] & 0 \\
\end{array}\right]
\left[\begin{array}{cc}
\Id_{r-k} & 0 \\
E' & \Id_{n-(r-k)} \\
\end{array}\right] \left( =\left[\begin{array}{cc}
(A')^{T}\JDiag[r-k]+\JDiag[r-k]E' & \JDiag[r-k] \\
\JDiag[n-(r-k)] & 0 \\
\end{array}\right] \right)$
\end{center}
must equal
\begin{center}
$\left[\begin{array}{cc}
0 & \JDiag[r-k] \\
\JDiag[n-(r-k)] & 0 \\
\end{array}\right]$
\end{center}
which implies that $E'=-\JDiag[r-k](A')^{T}\JDiag[r-k]$.\\
Any $y$ clearly satisfies ~\eqref{equation:StructureSP2nCE}.\\
And finally for $y$ to satisfy ~\eqref{equation:StructureSP2nAD} we must have
\begin{center}
$\left[\begin{array}{cc}
0 & \mathcal{D}_{2} \\
0 & -\JDiag[r-k](A')^{T}\JDiag[r-k]\mathcal{D}_{2} \\
\end{array}\right]^{T}
\left[\begin{array}{cc}
0 & \JDiag[n-(r-k)] \\
\JDiag[r-k] & 0 \\
\end{array}\right]
\left[\begin{array}{cc}
\Id_{r-k} & 0 \\
A' & \Id_{n-(r-k)} \\
\end{array}\right] \left( =\left[\begin{array}{cc}
0 & 0 \\
0 & \mathcal{D}_{2}^{T}\JDiag[n-(r-k)] \\
\end{array}\right] \right)$
\end{center}
must equal
\begin{center}
$\left[\begin{array}{cc}
\Id_{r-k} & 0 \\
A' & \Id_{n-(r-k)} \\
\end{array}\right]^{T}
\left[\begin{array}{cc}
0 & \JDiag[r-k] \\
\JDiag[n-(r-k)] & 0 \\
\end{array}\right]
\left[\begin{array}{cc}
0 & \mathcal{D}_{2} \\
0 & -\JDiag[r-k](A')^{T}\JDiag[r-k]\mathcal{D}_{2} \\
\end{array}\right] \left( =\left[\begin{array}{cc}
0 & 0 \\
0 & \JDiag[n-(r-k)]\mathcal{D}_{2} \\
\end{array}\right] \right)$
\end{center}
which implies that $\JDiag[n-(r-k)]\mathcal{D}_{2}=\mathcal{D}_{2}^{T}\JDiag[n-(r-k)]$, or equivalently $\JDiag[n-(r-k)]\mathcal{D}_{2}\in \Sym_{n-(r-k)}$.
\end{proof}

\begin{remark}
\label{remark:identificationOMinusPTP}
We may identify the two spaces $O_{P/\tilde{P}}^{-}$ and $O_{\GL_n/P'_{\widehat{r-k}}}^{-}$ under the map 
\begin{center}
$\left[\begin{array}{cc}
\mathcal{A} & 0 \\
0 & \JDiag[n](\mathcal{A}^{T})^{-1}\JDiag[n]
\end{array}\right]\mapsto \mathcal{A}$.
\end{center}
\end{remark}

\begin{remark}
\label{remark:identificationOfYPWasProduct}
Let $V_{w}$ be the linear subspace of $\Sym_n$ given by 
$x_{ij}=0$ if $j\leq r-k$ or $i<n-(r-k)$.
And let $V_{w}^{\prime}$ is the linear subspace of $O_{\GL_n/P'_{\widehat{r-k}}}^{-}$
given by $x_{ij}=0$ if $i>r$ and $j\leq r-k$.

\noindent Consider the map $\delta:Y_{\tilde{P}}(\tilde{w})\hookrightarrow O_{G/\tilde{P}}^{-} = O_{G/P}^{-}\times O_{P/\tilde{P}}^{-} \cong O_{G/P}^{-}\times O_{\GL_n/P'_{\widehat{r-k}}}^{-}$, where the first map is inclusion, the second is simply the product decomposition, and the final map is from Remark \ref{remark:identificationOMinusPTP}. This map is given explicitly by
\begin{center}
$\left[\begin{array}{cccc}
\Id_{r-k} & 0 & 0 & 0 \\
A' & \Id_{n-(r-k)} & 0 & 0 \\
0 & \mathcal{D}_{2} & \Id_{n-(r-k)} & 0 \\
0 & -\JDiag[r-k](A')^{T}\JDiag[r-k]\mathcal{D}_{2} & -\JDiag[r-k](A')^{T}\JDiag[r-k] & \Id_{r-k} \\
\end{array}\right] \mapsto \left(\left[\begin{array}{cccc}
\Id_{r-k} & 0 & 0 & 0 \\
0 & \Id_{n-(r-k)} & 0 & 0 \\
-\mathcal{D}_{2}A' & \mathcal{D}_{2} & \Id_{n-(r-k)} & 0 \\
\JDiag[r-k](A')^{T}\JDiag[r-k]\mathcal{D}_{2}A' & -\JDiag[r-k](A')^{T}\JDiag[r-k]\mathcal{D}_{2} & 0 & \Id_{r-k} \\
\end{array}\right],\left[\begin{array}{cc}
\Id_{r-k} & 0 \\
A' & \Id_{n-(r-k)} \\
\end{array}\right]\right) $
\end{center}
\noindent Consider the isomorphism $\gamma:O_{G/P}^{-}\times O_{\GL_n/P'_{\widehat{r-k}}}^{-}\rightarrow \Sym_n\times O_{\GL_n/P'_{\widehat{r-k}}}^{-}$ (cf. Remark \eqref{remark:identificationOMinusGP}) given by
\begin{center}
$\left(\left[\begin{array}{cc}
\Id_{n} & 0 \\
L & \Id_{n)} \\
\end{array}\right],
\left[\begin{array}{cc}
\Id_{r-k} & 0 \\
N & \Id_{n-(r-k)} \\
\end{array}\right]\right) \mapsto \left((LN)^{T}\JDiag[n]N,\left[\begin{array}{cc}
\Id_{r-k} & 0 \\
N & \Id_{n-(r-k)} \\
\end{array}\right]\right)$
\end{center}
We have that under the map $\gamma \circ \delta$, $Y_{\tilde{P}}(\tilde{w})$ gets identified with $V_w \times V'_w$. This follows by a simple computation and Corollary \ref{corollary:schubertVarietySP2nMatrixForm}. 
\end{remark}

\begin{definition}
\label{definition:ZPW}
Now let $Z_{\tilde{P}}(\tilde{w})\coloneqq Y_{P}(w)\times_{X_{P}(w)}X_{\tilde{P}}(\tilde{w})$.Then $Z_{\tilde{P}}(\tilde{w})=(O_{G/P}^{-}\times P/\tilde{P})\cap X_{\tilde{P}}(\tilde{w})$. Hence $Z_{\tilde{P}}(\tilde{w})$ is smooth, being open in the smooth $X_{\tilde{P}}(\tilde{w})$ (cf. Proposition \ref{proposition:schubertVarietySL2nSmooth}). 
\end{definition}

We define the map $p$ to be the composition of the inclusion map $Z_{\tilde{P}}(\tilde{w})\rightarrow O_{G/P}^{-}\times P/\tilde{P}$ and the projection map $O_{G/P}^{-}\times P/\tilde{P}\rightarrow P/\tilde{P}\;(\cong \GL_n/P'_{\widehat{r-k}})$. Then 
\begin{center}
$p(\left[\begin{array}{cc}
A_{\,} & 0\\
D_{\,} & \JDiag[n](A^{T})^{-1}\JDiag[n]
\end{array}\right](\textrm{mod}\,\tilde{P}))=A(\textrm{mod}\,P'_{\widehat{r-k}})$.
\end{center}
by Proposition \ref{proposition:structureOMinusSp2n}\eqref{proposition:structureOMinusSp2nInverseImage}\eqref{proposition:structureOMinusSp2nSLnModQ}.

We can induce an action of $B_{n} \subset \GL_n$ on $Z_{\tilde{P}}(\tilde{w})$ by the inclusion $B_n \rightarrow B_G$ given by
\begin{center}
$A \mapsto\left[\begin{array}{cc}
A & 0_{n\times n}\\
0_{n\times n} & \JDiag[n](A^{T})^{-1}\JDiag[n]
\end{array}\right]$.
\end{center}
Then the group $B_G$, and hence $B_n$, acts on $Z_{\tilde{P}}(\tilde{w})$ by left multiplication.

\begin{proposition}
\label{proposition:ZBnEquivariantInduced}
The space $Z_{\tilde{P}}(\tilde{w})$ is stable under the action of $B_n$ described above. For the same action, the map $p$ is $B_{n}$-equivariant.
\end{proposition}

\begin{proof}
Let $z\in \SP_{2n}$ such that $z\tilde{P}\in Z_{\tilde{P}}(\tilde{w})$. Then by Proposition \ref{proposition:structureOMinusSp2n}\eqref{proposition:structureOMinusSp2nInverseImage} we may write 
\begin{center}
$z=\left[\begin{array}{cc}
A_{\,} & 0\\
D_{\,} & \JDiag[n](A^{T})^{-1}\JDiag[n]
\end{array}\right](\textrm{mod}\,\tilde{P})$
\end{center}
with $A$ invertible by Proposition \ref{proposition:structureOMinusSp2n}\eqref{proposition:structureOMinusSp2ninvertibleA}. Then for $A'\in B_n$ we have that $A' \cdot z$ equals
\begin{center}
 $\left[\begin{array}{cc}
 A^{\prime} & 0_{n\times n}\\
 0_{n\times n} & \JDiag[n](A^{\prime T})^{-1}\JDiag[n]
 \end{array}\right]z=\left[\begin{array}{cc}
 A^{\prime}A & 0_{\,}\\
 \JDiag[n](A^{\prime T})^{-1}\JDiag[n]D\,\,\,\,\, & \JDiag[n](A^{\prime T})^{-1}(A^{\prime T})^{-1}\JDiag[n]
 \end{array}\right]$
 \end{center}
 Define $z':=A' \cdot z$. Proposition \ref{proposition:structureOMinusSp2n}\eqref{proposition:structureOMinusSp2ninvertibleA} implies that $z'$ is in $O_{G/P}^{-}\times P/\tilde{P}$ if and only if $A'A$ is invertible; clearly this is the case since both $A$ and $A'$ are invertible. The map $X_{B_{G}}(\tilde{w})\rightarrow X_{\tilde{P}}(\tilde{w})$ induced by $G / B_G \rightarrow G / \tilde{P}$ is surjective and thus $z (\textrm{mod}\,B_G) \in X_{B_{G}}(\tilde{w})$. Then $z' (\textrm{mod}\,B_G) \in X_{B_{G}}(\tilde{w})$ since the Schubert variety is a $B_G$-orbit closure, which implies $z' (\textrm{mod}\,\tilde{P}) \in X_{\tilde{P}}(\tilde{w})$. The space $Z_{\tilde{P}}(\tilde{w})=(O_{G/P}^{-}\times P/\tilde{P})\cap X_{\tilde{P}}(\tilde{w})$ and so $z' \in Z_{\tilde{P}}(\tilde{w})$. The $B_n$-equivariance of $p$ follows by $p(A' \cdot z) = A'A = A'p(z)$.
\end{proof}

\begin{theorem}
\label{theorem:subbundleSp2n}
With the notation developed in this section, and $w^{\prime}:=(k+1,..,r,n,..,r+1,k,..,1)\in S_n$, the Weyl group of $\GL_n$,
\begin{enumerate}
\item \label{enum:birationalDesingSp2n}
The natural map $X_{\tilde P}({\tilde{w}}) \rightarrow X_P(w)$ induced by the map $\,G/\tilde{P} \rightarrow G/P$ is proper and birational. This induces a proper and birational map $Z_{\tilde{P}}(\tilde w) \rightarrow Y_P(w)$ that is a desingularization of $Y_P(w)$.

\item \label{enum:birationalFibreSp2n}
The fiber of the desingularization $Z_{\tilde{P}}(\tilde w) \rightarrow Y_P(w)$ at $e_{id}$ is isomorphic to $X_{P'_{\widehat{r-k}}}(w')$.

\item \label{enum:FibrationSp2n}
The image $p(Z_{\tilde{P}}(\tilde w))$ equals $X_{P'_{\widehat{r-k}}}(w^{\prime})$. The map $p$ is a fibration and the fibers are isomorphic to $V_w$.

\item \label{enum:birationalVBRestrictionSp2n}
$Z_{\tilde{P}}(\tilde w)$ is a sub-bundle of the trivial bundle $O_{G/P}^{-} \times X_{P'_{\widehat{r-k}}}(w^{\prime})$. It is the restriction of the homogeneous vector bundle associated to the $P'_{\widehat{r-k}}$-module $V_w$ on $\GL_n/P'_{\widehat{r-k}}$. The module $V_w$ is a $P'_{\widehat{r-k}}$-submodule of $O_{G/P}^{-}$.

\end{enumerate}
\end{theorem}

\begin{proof}
\noindent\eqref{enum:birationalDesingSp2n}: The fact that $\tilde{w}$ is a coset representative of $w\tilde{P}$ in $W^{\tilde{P}}$ implies that the scheme-theoretic image of the composite map $X_{\tilde{P}}(\tilde{w})\hookrightarrow G/\tilde{P} \rightarrow G/P$ is $X_{P}(w)$. This map is clearly proper and the fact that $\tilde{w}$ is a minimal coset representative implies birationality. The above, combined with the definition of $Z_{\tilde{P}}(\tilde w)$, implies that $Z_{\tilde{P}}(\tilde w) \rightarrow Y_P(w)$ is a desingularization of $Y_P(w)$.

\noindent\eqref{enum:birationalFibreSp2n}: We have show in Remark \ref{remark:identificationOfYPWasProduct} that $Y_{\tilde{P}}({\tilde{w}})$ is identified with $V_w\times V_{w}^{\prime}$. Under the map $Y_{\tilde P}({\tilde{w}}) \rightarrow Y_P(w)$ the fiber at $e_{id}$ is $\{0\} \times V_{w}^{\prime}$. Our definition of $V_{w}^{\prime}$ implies that $p(\{0\} \times V_{w}^{\prime})$ corresponds to $Y_{P'_{\widehat{r-k}}}(w^{\prime})$ in $P/\tilde{P} \cong \GL_n/P'_{\widehat{r-k}}$. Thus the closure of $p(\{0\} \times V_{w}^{\prime})$ in $P/\tilde{P}$ is $X_{P'_{\widehat{r-k}}}(w^{\prime})$, which implies that the closure of $\{0\} \times V_{w}^{\prime}$ in $O_{G/P}^{-} \times P/\tilde{P}$ is precisely $\{0\} \times X_{P'_{\widehat{r-k}}}(w^{\prime})$. Our result follows by the fact that $Z_{\tilde P}({\tilde{w}})$ is the closure of $Y_{\tilde P}({\tilde{w}})$ in $O_{G/P}^{-} \times P/\tilde{P}$.

\noindent\eqref{enum:FibrationSp2n}: From Remark \ref{remark:identificationOfYPWasProduct} we have $p(Y_{\tilde P}({\tilde{w}}))=V_{w}^{\prime}\subseteq X_{P'_{\widehat{r-k}}}(w^{\prime})$. Since $Z_{\tilde{P}}({\tilde{w}})$ is the closure of $Y_{\tilde P}({\tilde{w}})$ we have that $Y_{\tilde P}({\tilde{w}})$ is dense in $Z_{\tilde{P}}({\tilde{w}})$. Further, the Schubert variety $X_{P'_{\widehat{r-k}}}(w^{\prime})$ is, by definition, closed in $\GL_n/P'_{\widehat{r-k}}$. Thus $p(Y_{\tilde{P}}({\tilde{w}}))\subseteq X_{P'_{\widehat{r-k}}}(w^{\prime})$ implies $p(Z_{\tilde{P}}({\tilde{w}}))\subseteq X_{P'_{\widehat{r-k}}}(w^{\prime})$. Part $\eqref{enum:birationalFibreSp2n}$ implies $X_{P'_{\widehat{r-k}}}(w^{\prime}) \subseteq p(Z_{\tilde{P}}({\tilde{w}}))$; thus $p(Z_{\tilde P}({\tilde{w}}))= X_{P'_{\widehat{r-k}}}(w^{\prime})$.

We now wish to show that the fibers of $p$ are isomorphic to $V_w$. Remark \ref{remark:identificationOfYPWasProduct} shows that $p^{-1}(e_{id}) \cong V_w$. In Proposition \ref{proposition:ZBnEquivariantInduced} we show that $p$ is $B_n$-equivariant. Further, each $B_n$-orbit of a point in $X_{P'_{\widehat{r-k}}}(w^{\prime})$ has non-trivial intersection with the opposite cell $Y_{P'_{\widehat{r-k}}}(w^{\prime})$, and hence non-trivial intersection with $V_{w}^{\prime}$. These combine to imply that any fiber of $p$ is isomorphic to the fiber at $e_{id}$, $p^{-1}(e_{id}) \cong V_w$.

\noindent\eqref{enum:birationalVBRestrictionSp2n}: Define a right action of $\GL_n$ on $O_{G/P}^{-}$(identified with $\Sym_n$ as in Remark ~\ref{remark:identificationOMinusGP}) as $g\circ v=g^{T}vg$ for $g\in \GL_n,\;v\in \Sym_n$. This induces an action of $P'_{\widehat{r-k}}$ on $O_{G/P}^{-}$ under which $V_w$ is stable. Thus there is an associated homogeneous bundle $\GL_n \times^{P'_{\widehat{r-k}}} V_w \to \GL_n/P'_{\widehat{r-k}}$. We will show that $Z_{\tilde{P}}(\tilde w))$ is isomorphic to the restriction of this bundle to the Schubert subvariety $X_{P'_{\widehat{r-k}}}(w^{\prime})$, thus completing our proof. Consider the following commutative diagram
\begin{center}
\begin{tikzcd}
Z_{\tilde{P}}(\tilde{w})
\arrow[drrr, "\phi"]
\arrow[ddrr, "p", swap]
\arrow[drr, dotted, "\psi" description] & & & \\
& & (\GL_{n}\times^{P'_{\widehat{r-k}}} V_w)|_{X_{P'_{\widehat{r-k}}}(w^{\prime})} \arrow[r] \arrow[d, "\pi'_{V_w}"]
& \GL_{n}\times^{P'_{\widehat{r-k}}} V_w \arrow[d, "\pi_{V_w}"] \\
& & X_{P'_{\widehat{r-k}}}(w^{\prime}) \arrow[r, "\iota"]
& \GL_n/P'_{\widehat{r-k}}
\end{tikzcd}
\end{center}
In this diagram, $\iota$ is the inclusion of the Schubert variety into the Grassmannian $\GL_n/P'_{\widehat{r-k}}$, $\pi_{V_w}$ is the bundle map defined in Section \ref{subsection:Preliminaries:homogVectorBundles}, and $\pi'_{V_w}$ is the restriction of $\pi_{V_w}$ to the subbundle. Define the map $\phi$ by
\begin{center}
$\phi :  {{\begin{bmatrix}
A & 0_{n \times n} \\
D & \JDiag[n](A^{T})^{-1}\JDiag[n]
\end{bmatrix}}(\textrm{mod}\,\tilde P)}
\longmapsto (A,D^T\JDiag[n]A)/\sim.$
\end{center}
Note that this map is well-defined by Remark \ref{remark:identificationOfYPWasProduct} and injective by Proposition \ref{proposition:structureOMinusSp2n}\eqref{proposition:structureOMinusSp2nInverseImage}. It is a trivial check to verify that $\iota \circ p = \pi_{V_w} \circ \phi$. Thus the universal property of products implies that $\psi$ exists, and it is injective. Finally, $Z_{\tilde{P}}(\tilde{w})$ has the same dimension as $(\GL_{n}\times^{P'_{\widehat{r-k}}} V_w)|_{X_{P'_{\widehat{r-k}}}(w^{\prime})}$ and thus $\psi$ is an isomorphism.
\end{proof}

\noindent As an immediate consequence of Theorem \ref{theorem:subbundleSp2n} we have

\begin{corollary}
\label{corollary:geometricTechniqueRealizationSp2n1}
For the open affine subvariety $Y_{P}(w)$ in $O_{G/P}^{-}$, we are able to realize the commutative diagram required for the geomtric technique as
\begin{center}
\begin{tikzcd}
Z_{\tilde P}(\tilde w) \arrow[hookrightarrow]{r} \arrow[d, "q'"]
& O^-_{G/P} \times X_{P'_{\widehat{r-k}}}(w^{\prime}) \arrow[d, "q"] \arrow[r] 
& X_{P'_{\widehat{r-k}}}(w^{\prime})\\
Y_{P}(w) \arrow[hookrightarrow]{r} 
& O^-_{G/P}
\end{tikzcd}
\end{center}
\end{corollary}

Computing the required vector bundle cohomology groups on the Schubert variety $X_{P'_{\widehat{r-k}}}(w^{\prime})$ would be quite difficult. Thankfully, we can simplify the picture somewhat.

\begin{proposition}
\label{proposition:smallerGrass}
$\;$

\begin{enumerate}
\item \label{proposition:smallerGrass:Ident}
The Schubert variety $X_{P'_{\widehat{r-k}}}(w^{\prime})$ is isomorphic to the Grassmannian $\GL_r/P''_{\widehat{r-k}}$, where $P''_{\widehat{r-k}}$ is the parabolic subgroup in $\GL_r$ omitting $\alpha_{r-k}$. 
\item \label{proposition:smallerGrass:VB}
$(\GL_{n}\times^{P'_{\widehat{r-k}}} V_w)|_{X_{P'_{\widehat{r-k}}}(w^{\prime})} \cong (\GL_{n}\times^{P'_{\widehat{r-k}}} V_w)|_{\GL_r/P''_{\widehat{r-k}}} \cong \GL_{r}\times^{P''_{\widehat{r-k}}} V_w$ as homogeneous vector bundles.
\end{enumerate}
\end{proposition}

\begin{proof}
$\;$

\noindent ~\eqref{proposition:smallerGrass:Ident}: This is clear.

\noindent ~\eqref{proposition:smallerGrass:VB}: Consider the embedding $i:\GL_r \xhookrightarrow{} \GL_n$ given by
\begin{center}
$R\mapsto \left[\begin{array}{cc}
R & 0\\
0 & \Id_{n-r}
\end{array}\right]$ .
\end{center}
Define the action of $\GL_r$ on $\Sym_n$ as the action induced by this embedding. This induces an action of $P''_{\widehat{r-k}}$ on $\Sym_n$. As $i(P''_{\widehat{r-k}})\subset P'_{\widehat{r-k}}$, the $P'_{\widehat{r-k}}$ stability of $V_w$ implies the $P''_{\widehat{r-k}}$ stability of $V_w$. Hence our result follows.
\end{proof}

\begin{corollary}
\label{corollary:geometricTechniqueRealizationSp2n2}
For the open affine subvariety $Y_{P}(w)$ in $O_{G/P}^{-}$, we are able to realize the commutative diagram required for the geomtric technique as
\begin{center}
\begin{tikzcd}
Z_{\tilde P}(\tilde w) \arrow[hookrightarrow]{r} \arrow[d, "q'"]
& O^-_{G/P} \times \GL_r/P''_{\widehat{r-k}} \arrow[d, "q"] \arrow[r] 
& \GL_r/P''_{\widehat{r-k}}\\
Y_{P}(w) \arrow[hookrightarrow]{r} 
& O^-_{G/P}
\end{tikzcd}
\end{center}
\end{corollary}

\section{Free Resolutions}
\label{sec:freeresolutions}
\subsection*{The geometric technique of Kempf-Lascoux-Weyman}%

We recall the geometric technique, due originally to Kempf, for computing minimal free resolutions. For the full details of the construction as well as proofs and related results see \cite[Chapter~5]{WeymanCohVBSyz03}.

Suppose that $Y$ is a closed subvariety of the affine variety $\mathbb{A}$. Let $\mathbb{C}[Y]$ and $R$ be their respective coordinate rings; $R$ is a polynomial ring with homogeneous maximal ideal $\mathfrak{m}$. Suppose there exists a projective variety $V$ such that the commutative diagram
\begin{center}
\begin{tikzcd}
Z \arrow[hookrightarrow]{r} \arrow[d, "q'"]
& \mathbb{A}  \times V \arrow[d, "q"] \arrow[r] & V\\
Y \arrow[hookrightarrow]{r} 
& \mathbb{A}
\end{tikzcd}
\end{center}
may be constructed, with $q$ being first projection and its restriction $q'$ being a proper, birational map from a subbundle $Z$ of the trivial bundle $\mathbb{A} \times V$. Denote the dual of the quotient bundle on $V$ associated to the subbundle $Z$ by $\xi$. Then the following theorem holds.

\begin{theorem}[\protect{\cite[Theorem~5.1.2]{WeymanCohVBSyz03}}]
\label{theorem:geometrictechnique}
There is a finite complex $(F_\bullet,
\partial_\bullet)$ of finitely generated graded free $R$-modules that is
quasi-isomorphic to the right derived image $\RDer q'_* \mathscr{O}_{Z}$, with \[
F_i = \bigoplus_{j \geq 0} \homology^j(V, \bigwedge^{i+j} \xi)
\otimes_\mathbb{C} R(-i-j),
\]
and $\partial_i(F_i) \subseteq \mathfrak{m} F_{i-1}$.  Furthermore, the following
are equivalent:
\begin{enumerate}
\item $Y$ has rational singularities, that is  $\RDer q'_* \mathscr{O}_{Z}$ is 
quasi-isomorphic to the structure sheaf $\mathscr{O}_{Y}$;
\item $F_\bullet$ is a minimal $R$-free resolution of $\mathbb{C}[Y]$, that is, 
$F_0 \simeq R$ and $F_{-i} = 0$ for every $i > 0$.
\end{enumerate}
\end{theorem}

A sketch of this proof is given in \cite[Section 4]{KLPSResolutionSchSing15}, and ~\cite[5.1.3]{WeymanCohVBSyz03} may be consulted for a more comprehensive account.

We would like to use Theorem \ref{theorem:geometrictechnique} to give a minimal free resolution for the opposite cells $Y_P(w)$, $w = \mathcal{W}_{k,r}$ for $1 \leq k < r \leq n$. We have constructed the required commutative diagram in Corollary \ref{corollary:geometricTechniqueRealizationSp2n2} and combining this with Theorem \ref{theorem:geometrictechnique} we have the following result.

\begin{theorem}
\label{theorem:stepone}
Write $\xi$ for the homogeneous vector bundle on $\GL_r/P''_{\widehat{r-k}}$ 
associated to the $P''_{\widehat{r-k}}$-module $\left(O^-_{G/P}/ V_w\right)^*$
(this is the dual of the quotient of $O^-_{G/P} \times \GL_r/P''_{\widehat{r-k}}$ by $Z_{\tilde{P}}(\tilde w)$). Let $R$ be the coordinate ring of $O^-_{G/P}$. Then the complex $(F_\bullet, \partial_\bullet)$ is a minimal $R$-free resolution of $\mathbb{C}[Y_P(w)]$ given by
\[
F_i = \oplus_{j \geq 0} \homology^j(\GL_r/P''_{\widehat{r-k}}, \bigwedge^{i+j} \xi)
\otimes_\mathbb{C} R(-i-j).
\]
\end{theorem}

Computing the cohomology groups found in Theorem~\ref{theorem:stepone} 
is a difficult problem. Techniques for computing them are discussed in the following section.
\section{Cohomology of Homogeneous Vector-Bundles}
\label{sec:cohomologyHVBStepTwo}

We have shown in Theorem ~\ref{theorem:stepone} that the calculation of a minimal $R$-free resolution of $\mathbb{C}[Y_P(w)]$ comes down to the computation of the cohomology of certain homogeneous bundles over $\GL_r/P''_{\widehat{r-k}}$. In particular we need to calculate 
\begin{equation}
\label{equation:cohomologyVB}
\homology^{\bullet}(\GL_r/P''_{\widehat{r-k}}, \bigwedge^{t} \xi)
\end{equation}
for arbitrary $t$.

The $P''_{\widehat{r-k}}$-module $\left(O^-_{G/P}/ V_w\right)^*$ is not completely reducible (the unipotent radical of $P''_{\widehat{r-k}}$ does not act trivially), and thus we can not use the Bott-algorithm to compute its cohomology. In \cite{OttavianiRubeiQuiversCohVB06} the authors determine the cohomology of general homogeneous bundles on Hermitian symmetric spaces. As $\GL_r/P''_{\widehat{r-k}}$ is such a space their results could be used to determine \eqref{equation:cohomologyVB}. In practice proceeding along these lines is possible though extremely complicated.

Another approach to the calculation of these cohomologies comes from using a technique employed in \cite[Chapter 6.3]{WeymanCohVBSyz03}. There the minimal $R$-free resolution of a related space is computed and the minimal $R$-free resolution of $\mathbb{C}[Y_P(w)]$ can be seen as a sub-resolution. In \cite{WeymanCohVBSyz03} this method is used for the case when $n=r$. That is, the case where $Y_P(w)$ is the symmetric determinantal variety. In this case the authors assume that $k=2u$(the odd case can be reduced to this even case). They look at the subspace $T_w$ of $\Sym_n$ given by symmetric matrices of block form
\begin{center}
$\left[\begin{array}{cc}
0_{n-u \times n-u} & R\\
R^{T} & S_{u \times u}
\end{array}\right]$
\end{center}
Let $P'_{\widehat{n-u}}$ be the parabolic subgroup of $\GL_n$ omitting the root $\alpha_{n-u}$, then $T_w$ is a $P'_{\widehat{n-u}}$-module under the same action. If $Z_w$ is the homogeneous vector bundle associated with $T_w$ we have the following diagram
\begin{center}
\begin{tikzcd}
Z_w \arrow[hookrightarrow]{r} \arrow[d, "q'"]
& \Sym_n  \times \GL_n/P'_{\widehat{n-u}} \arrow[d, "q"] \arrow[r] & \GL_n/P'_{\widehat{n-u}}\\
Y \arrow[hookrightarrow]{r} 
& \Sym_n
\end{tikzcd}
\end{center}
They show that the resolution of $\mathbb{C}[Y_P(w)]$ can be realized as a sub-resolution of the resolution of $\mathbb{C}[Y]$. In this case, the $P'_{\widehat{n-u}}$-module $(\Sym_n/T_w)^{*}$(this is the dual of the quotient of $\Sym_n  \times \GL_n/P'_{\widehat{n-u}}$ by $Z_w$) is completely reducible and thus the cohomology of the corresponding homogeneous vector bundles $\bigwedge^{t} \xi$  may be computed using the Bott-algorithm, leading to
\begin{theorem}
\cite[Theorem 6.3.1(c)]{WeymanCohVBSyz03} The $i$th term $G_i$ of the minimal free resolution of $\mathbb{C}[Y_P(w)]$ as an $R$ module is given by the formula 
\begin{center}
$G_i=\bigoplus\limits_{\substack{\lambda \in Q_{k-1}(2t) \\ \mathrm{rank}\lambda\;\mathrm{even} \\ i=t-k\frac{1}{2}\mathrm{rank}\lambda}} \Schur_{\chk{\lambda}}\mathbb{C}^{n} \otimes_{\mathbb{C}} R$ 
\end{center}

\end{theorem}
\noindent Here $Q_{k-1}(2t)$ is the set of partitions $\lambda$ of $2t$ which in hook notation can be written as $\lambda = (a_1, . . . , a_{s} |b_1, . . . , b_{s} )$, where $s$ is a positive integer, and for each $j$ we have $a_j = b_j + (k-1)$. And $\chk{\lambda}$  is the conjugate (or dual) partition of $\lambda$. And finally, rank$\lambda$ is defined as being equal to $l$, where the largest square fitting inside $\lambda$ is of size $l \times l$.

Similar methods may be used to compute a closed form formula for the minimal free resolution of $\mathbb{C}[Y_P(w)]$ as an $R$ module in the case $r\neq n$.

\section*{Acknowledgements}
The authors thank the referee for some useful comments. In addition, the authors thank Manoj Kummini for helpful discussions.


\end{document}